\documentclass[a4paper,notitlepage,twoside,reqno,10pt]{amsart}

\usepackage{bbm,pifont,latexsym,dcolumn,indentfirst}
\usepackage[hypertexnames=false,bookmarksopen=true,linktocpage=true,pdfstartview={XYZ null null 1.25}]{hyperref}
\usepackage{amsthm,amsmath,amssymb,amscd,amsfonts,mathrsfs}
\usepackage{color,graphicx,xcolor,graphics,subfigure,extarrows,caption2}
\usepackage{pdfpages,titletoc}

\usepackage{autonum}

\newtheorem{thm}{Theorem}[section]
\newtheorem{lem}[thm]{Lemma}
\newtheorem{cor}[thm]{Corollary}
\newtheorem{prop}[thm]{Proposition}

\newtheorem*{mainthm}{Main Theorem}

\theoremstyle{definition}

\newtheorem*{ques}{Question}
\newtheorem*{claim}{Claim}

\newcommand{\D}{\mathbb{D}}

\newcommand{\R}{\mathbb{R}}
\newcommand{\Z}{\mathbb{Z}}
\newcommand{\N}{\mathbb{N}}
\newcommand{\Q}{\mathbb{Q}}

\newcommand{\C}{\mathbb{C}}
\newcommand{\EC}{\widehat{\mathbb{C}}}

\newcommand{\A}{\mathbb{A}}

\newcommand{\Crit}{\textup{Crit}}
\newcommand{\Int}{\textup{int}}
\newcommand{\Ext}{\textup{ext}}
\newcommand{\diam}{\textup{diam}}

\newcommand{\ii}{\textup{i}}
\newcommand{\id}{\textup{id}}
\newcommand{\re}{\textup{Re\,}}

\newcommand{\area}{\textup{area}}

\makeatletter\@addtoreset{equation}{section}\makeatother

\begin{document}

\author[Y. Fu]{Yuming Fu}
\address{College of Mathematics and Statistics, Shenzhen University, Shenzhen 518061, P. R. China}
\email{yumingfuxy@szu.edu.cn}

\author[F. Yang]{Fei Yang}
\address{Department of Mathematics, Nanjing University, Nanjing 210093, P. R. China}
\email{yangfei@nju.edu.cn}

\author[G. Zhang]{Gaofei Zhang}
\address{Department of Mathematics, Nanjing University, Nanjing 210093, P. R. China}
\email{zhanggf@nju.edu.cn}

\title[Quadratic rational maps with Siegel disks]{Quadratic rational maps with a $2$-cycle of Siegel disks}

\begin{abstract}
For the family of quadratic rational functions having a $2$-cycle of bounded type Siegel disks, we prove that each of the boundaries of these Siegel disks contains at most one critical point. In the parameter plane, we prove that the locus for which the boundaries of the $2$-cycle of Siegel disks contain two critical points is a Jordan curve.
\end{abstract}



\date{\today}



\maketitle


\section{Introduction}\label{introduction}

The hyperbolic components of rational maps have been studied a lot in the past $3$ decades, and one important study object is their boundedness. For quadratic rational maps,
Rees divided the hyperbolic components into $4$ types (see also \cite{Mil93}), and proved the boundedness of certain $1$-dimension loci in hyperbolic components \cite{Ree90}.
Epstein proved that the hyperbolic component of quadratic rational maps possessing two distinct attracting cycles is bounded if and only if neither attractor is a fixed point \cite{Eps00}.

To understand the structures of the hyperbolic component of quadratic rational maps, one usually studies some specific $1$-dimensional slices. Some other examples in this direction include \cite{AY09} and \cite{DFGJ14} etc. Moreover, understanding the maps on the boundaries of hyperbolic components is very helpful to grasp the characterizations of the hyperbolic components. See \cite{DeM05} and \cite{DeM07}.

As a partial complement of Epstein's study in \cite{Eps00}, a meaningful problem is to consider the quadratic rational maps possessing only one attracting cycle of period $2$ for which each of the immediate attracting basin contains exactly one critical point. A natural question is, what will happen if the modulus of the multiplier of the $2$-cycle attracting orbit tends to $1$? A little bit similar situation has been considered in \cite{BEE13} for the rationally indifferent case. In this paper, we study the limit parameter slices of irrationally indifferent case and hope that the main results can shed some lights on the study of hyperbolic components of quadratic rational maps.

\medskip

Let $f$ be a holomorphic function. A periodic Fatou component $U$ of $f$ is called a Siegel disk if there exists a minimal integer $p\geq 1$, such that $f^{\circ p}:U\to U$ is conformally conjugate to the irrational rotation $R_\theta(\zeta)=e^{2\pi\ii\theta}\zeta$ for some $\theta\in\R\setminus\Q$. The collection $\{U,f(U), \cdots, f^{\circ (p-1)}(U)\}$ is called a $p$-cycle of Siegel disks of $f$. In particular, $U$ contains a periodic point $z_0$ satisfying $(f^{\circ p})'(z_0)=e^{2\pi\ii\theta}$, which is called a Siegel point. If $p=1$, then $U$ is called a fixed Siegel disk and $z_0$ is a fixed Siegel point.

In the dynamical plane, the topology of the boundaries of Siegel disks has been studied extensively. They are motivated by the prediction of Douady and Sullivan that every Siegel disk of a rational map with degree at least two is a Jordan domain. We refer to \cite{PZ04}, \cite{Zak10}, \cite{Zha11}, \cite{SY21}, \cite{Che22b} and the references therein for corresponding results.

In the parameter plane, the results relating to the Siegel disks are much less. Zakeri considered the cubic polynomials having a fixed bounded type Siegel disk, and proved that the locus for which the fixed Siegel disk contains the both critical points in the boundary is a Jordan curve  \cite{Zak99}. One may refer to \cite{KZ09} and \cite{Yan13} for the similar results in the parameter spaces of some other holomorphic families containing a fixed Siegel disk.

\medskip
In this paper, we are interested in a special slice in the space of quadratic rational maps, they are maps containing a $2$-cycle of bounded type Siegel disks. In some sense, these maps can be seen as the limit maps of the quadratic rational maps containing a $2$-cycle of geometrically attracting periodic points as the multiplier tends to the point on the unit circle with irrational angle.

Let $\theta$ be an irrational number of bounded type, i.e., the continued fraction expansion $\theta=[a_0;a_1,a_2,\cdots,a_n,\cdots]$ satisfies $\sup_n \{a_n\}<+\infty$. Suppose that $f$ is a quadratic rational map having a $2$-cycle of Siegel disks with rotation number $\theta$. By a direct calculation, $f$ is conformally conjugate to a map in the following family (see  \S\ref{sec:one}):
\begin{equation}\label{equ:family}
\Sigma_\theta:=\left\{f_{\alpha}(z)=\alpha\,\frac{1+e^{2\pi\ii\theta} z}{z+z^2}:~\alpha\in \C\setminus\{0\} \right\}.
\end{equation}
On the other hand, one may verify that for any $\alpha\in\C\setminus\{0\}$, $f_\alpha\in\Sigma_\theta$ has a $2$-cycle of Siegel disks $\{\Delta_{\alpha}^0,\Delta_{\alpha}^{\infty}\}$ which contains the $2$-cycle $\{0,\infty\}$ whose multiplier is $\lambda=e^{2\pi\ii\theta}$ (see \S\ref{sec:one}). Every $f_\alpha$ has exactly two different critical points which are independent of $\alpha$:
\begin{equation}
\{c_1,c_2\} =\left\{-(1+\sqrt{1-\lambda})^{-1},~-(1-\sqrt{1-\lambda})^{-1}\right\}.
\end{equation}
Here we do not specify which critical point is $c_1$ temporarily. According to \cite{Zha11}, $\partial\Delta_{\alpha}^0$ and $\partial\Delta_{\alpha}^{\infty}$ are quasicircles, and $\partial\Delta_{\alpha}^0\cup\partial\Delta_{\alpha}^{\infty}$ contains at least one critical point $c_1$ or $c_2$ (see also \cite{GS03}).

\medskip
For a Jordan curve $\Gamma$ in $\C$, we use $\Gamma_{\Ext}$ to denote the unbounded component of $\C\setminus\Gamma$ and use $\Gamma_{\Int}$ to denote the bounded. In this paper, we prove the following main result (see Figure \ref{Fig_parameter}).

\begin{mainthm}\label{maintheorem}
For any bounded type $\theta$, each of $\partial\Delta_\alpha^0$ and $\partial\Delta_\alpha^\infty$ contains at most one critical point for all $f_\alpha\in\Sigma_\theta$, and there exist a marking $c_1=c_1(\theta)$, $c_2=c_2(\theta)$ and a Jordan curve $\Gamma=\Gamma_\theta\subset\Sigma_\theta$ separating $0$ from $\infty$ such that
\begin{enumerate}
\item If $\alpha\in\Gamma$, then $c_1\in\partial\Delta_{\alpha}^0$ and $c_2\in\partial\Delta_{\alpha}^\infty$;
\item If $\alpha\in\Gamma_{\Ext}$, then $c_1\in\partial\Delta_{\alpha}^0$ while $\partial\Delta_{\alpha}^\infty$ contains no critical point; and
\item If $\alpha\in\Gamma_{\Int}\setminus\{0\}$, then $c_2\in \partial\Delta_{\alpha}^\infty$ while $\partial\Delta_{\alpha}^0$ contains no critical point.
\end{enumerate}
Moreover, $\partial\Delta_{\alpha}^0$ and $\partial\Delta_{\alpha}^\infty$ depend continuously on $f_\alpha\in\Sigma_\theta$.
\end{mainthm}

\begin{figure}[!htpb]
 \setlength{\unitlength}{1mm}
  \centering
  \includegraphics[width=0.58\textwidth]{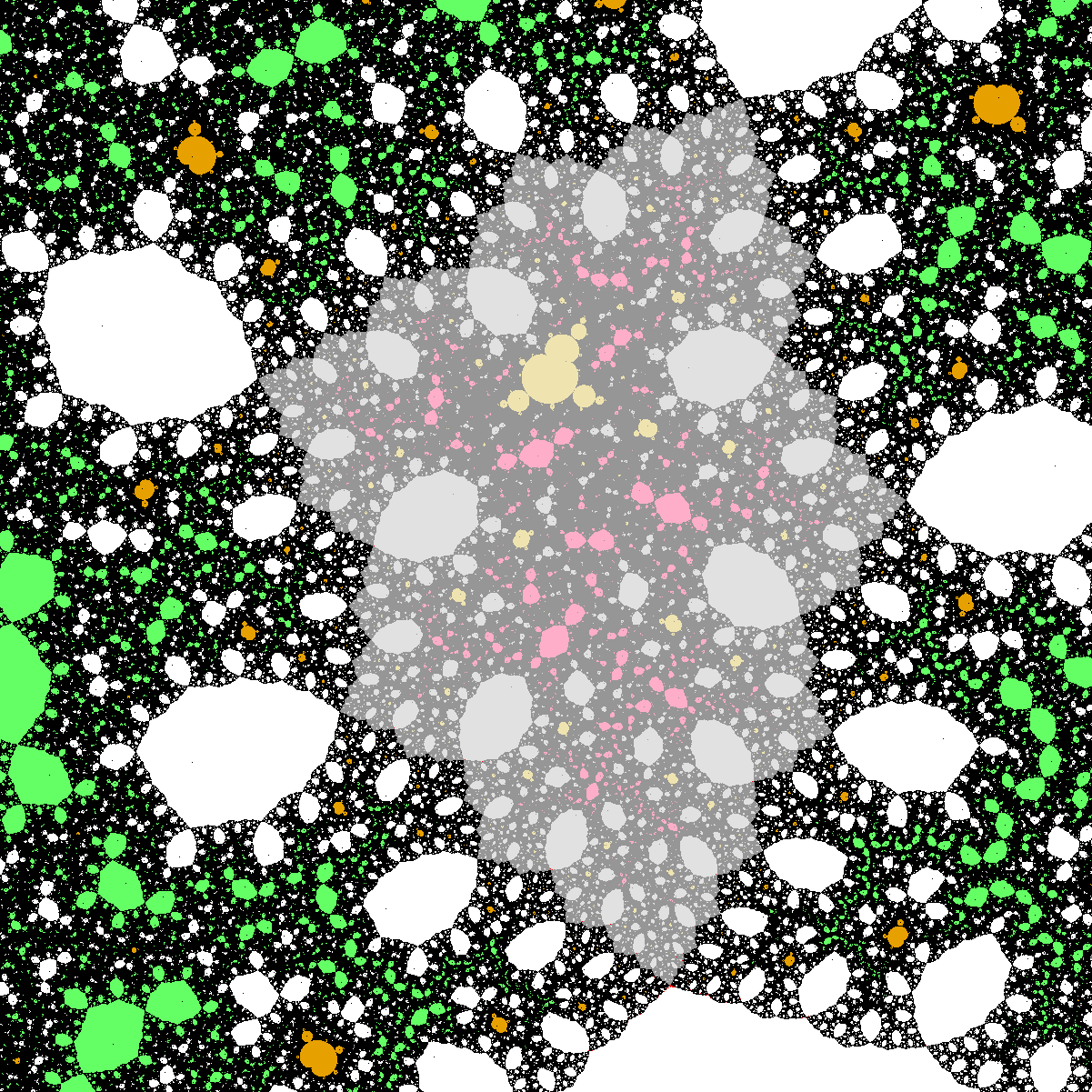}
  \caption{The parameter plane of $\Sigma_\theta$ with $\theta=(\sqrt{5}-1)/2$. One may observe the Jordan curve $\Gamma$ clearly, which is the common boundary of the deeply and lightly colored parts. Some Siegel capture components and baby Mandelbrot sets can be seen also. Figure range: $[-2,2]\times [-2.7, 1.3]$.}
  \label{Fig_parameter}
\end{figure}

In fact, we can clarify the marking of the critical points in the Main Theorem for some specific $\theta$'s (see \S\ref{subsec:theta}). In particular, when $\theta\in(\tfrac{1}{6},\tfrac{5}{6})$ is of bounded type, then the marking is
\begin{equation}\label{equ:c-1-2}
c_1 =-(1+\sqrt{1-\lambda})^{-1} \text{\quad and\quad} c_2=-(1-\sqrt{1-\lambda})^{-1},
\end{equation}
where the square root branch is chosen as $\re\sqrt{1-\lambda}>0$.
We believe that the above marking holds for all bounded type irrational numbers. However, our method in this paper cannot imply this, especially when $\theta$ is very close to $0$ (i.e., when $\lambda$ is very close to $1$). In this situation, the two critical points $c_1$, $c_2$ are very close and $f_\alpha$ can be seen as a small perturbation of $z\mapsto \alpha/z$, whose second iteration is the identity. Therefore, we raise the following:

\begin{ques}
Does $-(1+\sqrt{1-\lambda})^{-1}\in \partial\Delta_{\alpha}^0$ ($\alpha\in\Gamma_\theta$) hold for all bounded type $\theta$?
\end{ques}

This paper is organized as following:

In \S\ref{sec:one} we study the location of the critical points and prove the continuous dependence of $\partial\Delta_\alpha^0$ and $\partial\Delta_\alpha^\infty$ on $\alpha\in\Sigma_\theta$. In particular, we prove that each of $\partial\Delta_\alpha^0$ and $\partial\Delta_\alpha^\infty$ contains at most one critical point (Lemma \ref{lem:only-one}) and the key point is to exclude a case by studying the Thurston obstruction.

In \S\ref{limitdynamics}, we study the limit dynamics of $f_\alpha^{\circ 2}$ as $\alpha$ tends to the singularities $0$ and $\infty$, and give a compactification of $\Sigma_\theta$. We prove that $\partial\Delta_\alpha^0\cup\partial\Delta_\alpha^\infty$ contains exactly one critical point if $\alpha$ is sufficiently large or small. In particular, there exists a marking of critical points $c_1=c_1(\theta)$ and $c_2=c_2(\theta)$ such that $c_1\in\partial\Delta_\alpha^0$ if $\alpha$ is large and $c_2\in\partial\Delta_\alpha^\infty$ if $\alpha$ is small (Corollary \ref{cor:Siegel-crit}). We also prove that the marking can be chosen as in \eqref{equ:c-1-2} for some specific $\theta$'s (Proposition \ref{prop:bounded}).

In \S\ref{gamma}, we analyze the combinations of the Fatou components of $f_\alpha$ and prove that the Julia set of $f_\alpha$ has zero Lebesgue measure for all $\alpha\in\Gamma$, where $\Gamma\subset\Sigma_\theta$ is the locus of $\alpha$ such that $\partial\Delta_\alpha^0\cup\partial\Delta_\alpha^\infty$ contains both of the critical points.

In \S\ref{proof} we prove that $\Gamma$ is a Jordan curve by an argument of rigidity. For any $\alpha\in\Gamma$, the combination of $f_\alpha$ is determined by the conformal angle $A(\alpha)$ of the two critical points $c_1$ and $c_2$. A key ingredient in the proof of rigidity is to show that two maps in $\Gamma$ are Thurston equivalent if they have the same conformal angle (Lemma \ref{injection}).

\medskip
For the study of the parameter spaces of the holomorphic families containing a fixed Siegel disk, one may also refer to \cite{BH01}, \cite{BF10}, \cite{BCOT21}, \cite{Zak18} and \cite{Che20B}.

\medskip
\noindent\textbf{Acknowledgements.} The authors would like to thank Arnaud Ch\'{e}ritat for providing an algorithm to draw Figure \ref{Fig_parameter}, and the referee for very insightful and detailed comments and corrections.
This work was supported by NSFC (grant Nos.\,12071210, 12171276) and NSF of Jiangsu Province (grant No.\,BK20191246).

\section{Cycle of Siegel disks}\label{sec:one}

In this section, we first prove that $\Sigma_\theta$ is the desired family, and then prove some basic properties of the $2$-cycle of Siegel disks $\{\Delta_\alpha^0,\Delta_\alpha^\infty\}$. Finally we prove that $\partial\Delta_\alpha^0\cup\partial\Delta_\alpha^\infty$ moves continuously as $\alpha$ varies in $\Sigma_\theta$.

\subsection{The desired family}

In the following, the irrational number $\theta$ is assumed to be of bounded type, and we identify the parameter space $\Sigma_\theta$ with $\C\setminus\{0\}$.

\begin{lem}\label{sigma}
Let $f$ be a quadratic rational map having a $2$-cycle of Siegel disks with rotation number $\theta$. Then $f$ is conformally conjugate to $f_\alpha\in\Sigma_\theta$ for some $\alpha\in\C\setminus\{0\}$. Moreover, every $f_\alpha\in\Sigma_\theta$ has a $2$-cycle of Siegel disks $\{\Delta_\alpha^0,\Delta_\alpha^\infty\}$ containing the $2$-cycle $\{0,\infty\}$ whose multiplier is $e^{2\pi\ii\theta}$.
\end{lem}

\begin{proof}
Up to a conformal conjugacy, we assume that $f$ has a $2$-cycle of Siegel points $\{0,\infty\}$ of rotation number $\theta$ and that $-1$ is the other pole of $f$. Then $f$ can be written as
\begin{equation}
f(z)=a\,\frac{1+b z}{z+z^2},
\end{equation}
for some $a,b\in\C\setminus\{0\}$ and $b\neq 1$. Since $f(z)\sim\frac{a}{z}$ near $0$ and $f(z)\sim\frac{ab}{z}$ near $\infty$, we have
$(f\circ f)'(0)=b=e^{2\pi\ii\theta}$.
This implies that $f$ is conformally conjugate to $f_\alpha\in\Sigma_\theta$ for $\alpha=a$.

For the second statement, we have $f_\alpha(z)\sim\frac{\alpha}{z}$ near $0$ and $f(z)\sim\frac{\alpha e^{2\pi\ii\theta}}{z}$ near $\infty$ for any $\alpha\in\C\setminus\{0\}$. Hence
$(f_\alpha^{\circ 2})'(0)=e^{2\pi\ii\theta}$ and $f_\alpha$ has a $2$-cycle of Siegel disks with rotation number $\theta$.
\end{proof}

\subsection{The location of the critical points}

The main aim in this subsection is to prove that each of $\partial\Delta_\alpha^0$ and $\partial\Delta_\alpha^\infty$ cannot contain two critical points.

\begin{lem}\label{atleast1}
For any $f_\alpha\in\Sigma_\theta$, $\partial{\Delta_{\alpha}^0} \cup \partial{\Delta_{\alpha}^{\infty}}$ contains at least one critical point $c_1$ or $c_2$.
\end{lem}

\begin{proof}
Note that $\Delta_\alpha^0$ and $\Delta_\alpha^\infty$ are fixed Siegel disks of $f_\alpha^{\circ 2}$ and the rotation numbers are both $\theta$, which is of bounded type.  The set of critical points of $f_\alpha^{\circ 2}$ is
\begin{equation}\label{equ:crit-f2}
\Crit(f_\alpha^{\circ 2})=\{c_1,c_2\}\cup f_\alpha^{-1}(c_1)\cup f_\alpha^{-1}(c_2).
\end{equation}
According to \cite{GS03} or \cite{Zha11}, each of $\partial\Delta_{\alpha}^0$ and $\partial{\Delta_{\alpha}^{\infty}}$ must contain a point in $\Crit(f_\alpha^{\circ 2})$. In particular, this implies that $(\partial\Delta_{\alpha}^0\cup\partial\Delta_{\alpha}^\infty) \cap \{c_1,c_2\}\neq\emptyset$.
\end{proof}

The following result is important for us to locate the critical points and to prove the continuous dependence of $\partial\Delta_{\alpha}^0\cup\partial\Delta_{\alpha}^\infty$ on $\alpha$. See \cite{Zha11} for a proof.

\begin{thm}\label{zhang11}
Let $d\ge 2$ be an integer and $0<\theta<1$ be an irrational number of bounded type. Then there exists a constant $1<K(d,\theta)<\infty$ depending only on $d$ and $\theta$ such that for any rational map $f$ of degree $d$, if $f$ has a fixed Siegel disk with rotation number $\theta$, then the boundary of the Siegel disk is a $K(d,\theta)$-quasicircle which passes through at least one critical point of $f$.
\end{thm}

As an immediate consequence, we have

\begin{lem}\label{disjoint}
There exists $K=K(\theta)>1$ such that for any $f_\alpha\in\Sigma_\theta$, $\partial\Delta_{\alpha}^0$, $\partial\Delta_{\alpha}^\infty$ are $K$-quasicircles and $\overline{\Delta_\alpha^{0}}\cap \overline{\Delta_\alpha^{\infty}}=\emptyset$.
\end{lem}

\begin{proof}
By Theorem \ref{zhang11}, $\partial\Delta_{\alpha}^0$ and $\partial\Delta_{\alpha}^\infty$ are $K$-quasicircles (where $K>1$ is independent of $\alpha$) since they are the boundaries of two fixed bounded type Siegel disks of $f_\alpha^{\circ 2}$. The map $f^{\circ 2}_{\alpha}$ is topologically conjugate to the irrational rotation $R_\theta(\zeta)=e^{2\pi\ii\theta}\zeta$ on $\overline{\Delta_\alpha^{0}}$ and $\overline{\Delta_\alpha^{\infty}}$. Assume that $\partial\Delta_\alpha^{0}\cap \partial\Delta_\alpha^{\infty}$ contains a point $z_0$. Then the closure of the orbit $\{f^{\circ 2n}_{\alpha}(z_0):n\in\N\}$ is the common boundary of $\overline{\Delta_\alpha^{0}}$ and $\overline{\Delta_\alpha^{\infty}}$. This implies that the Julia set of $f_\alpha^{\circ 2}$ is empty, which is impossible.
\end{proof}

We now prove that any of $\partial\Delta_{\alpha}^0$ and $\partial{\Delta_{\alpha}^{\infty}}$ cannot contain two critical points.

\begin{lem}\label{lem:only-one}
For any $f_{\alpha}\in\Sigma_\theta$, if $\partial{\Delta_{\alpha}^0} \cup \partial{\Delta_{\alpha}^{\infty}}$ contains both of $c_1$ and $c_2$, then one of them is contained in $\partial{\Delta_{\alpha}^0}$ while the other is contained in $\partial{\Delta_{\alpha}^{\infty}}$.
\end{lem}

\begin{proof}
Suppose that $\partial\Delta_\alpha^0$ contains two critical points $c_1$ and $c_2$ (the proof for $\partial\Delta_\alpha^\infty$ is similar).
Note that $\deg(f_\alpha)=2$, $c_1\neq c_2$ and $f_{\alpha}: \Delta_{\alpha}^{0} \to \Delta_{\alpha}^{\infty}$ is conformal. It implies that $f^{-1}_{\alpha}(\Delta_{\alpha}^{\infty})$ consists of exactly two distinct components $\Delta_{\alpha}^{0}$ and $V_\alpha^0$, where $V_\alpha^0$ is a simply connected Fatou component attached to $\Delta_\alpha^0$ at $c_1$ and $c_2$ since $f_\alpha$ is locally two-to-one at $c_1$ and $c_2$. By Lemma \ref{disjoint}, $V_\alpha^0$ is a Jordan domain.

We claim that $\overline{V_\alpha^0}\cap \overline{\Delta_{\alpha}^\infty}=\emptyset$ and $\overline{V_\alpha^0}\cap \overline{\Delta_{\alpha}^0}=\{c_1,c_2\}$.
Indeed, if $\partial{V_\alpha^0}\cap \partial{\Delta_{\alpha}^\infty}$ contains a point $z_0$, then $f_\alpha(z_0)\in \partial\Delta_\alpha^{0}\cap \partial\Delta_\alpha^{\infty}$, which contracts Lemma \ref{disjoint}.
Any $z_1\in \partial{V_\alpha^0}\cap \partial{\Delta_{\alpha}^0}$ is a critical point of $f_\alpha$ since $f_\alpha(\Delta_\alpha^0)=f_\alpha(V_\alpha^0)=\Delta_\alpha^\infty$. Hence $\overline{V_\alpha^0}\cap \overline{\Delta_{\alpha}^0}=\{c_1,c_2\}$.
Since the critical values $f_\alpha(c_1)$ and $f_\alpha(c_2)$ are contained in $\partial\Delta_\alpha^\infty$, it implies that $f_\alpha^{-1}(\EC\setminus\overline{\Delta_\alpha^\infty})$ consists of two Jordan domains $U_\alpha^0$ and $U_\alpha^\infty$, where $U_\alpha^\infty$ is the disjoint union of $\overline{\Delta_\alpha^\infty}$ and the annulus $A_\alpha$. See Figure \ref{case2}.

\begin{figure}[!htpb]
 \setlength{\unitlength}{1mm}
  \centering
  \includegraphics[width=0.5\textwidth]{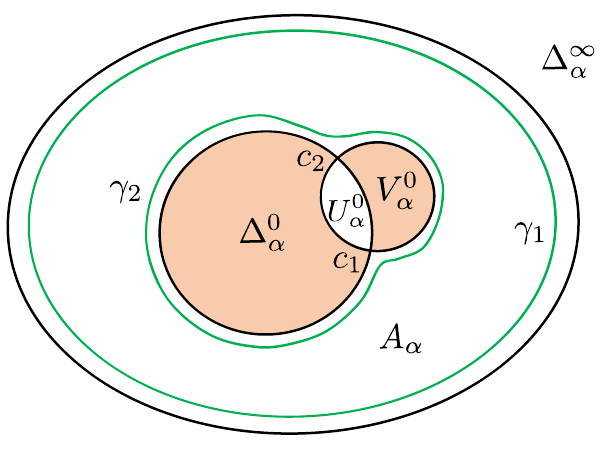}
  \caption{A candidate of the location of the critical points and some components are marked. We shall prove that this configuration is impossible for any $f_\alpha\in\Sigma_\theta$.}
  \label{case2}
\end{figure}

Let $\gamma_1$ be a Jordan curve in $A_\alpha$ which is sufficiently close to $\partial\Delta_\alpha^\infty$ and separates $\partial\Delta_\alpha^\infty$ from $\partial\Delta_\alpha^0$.
Since the post-critical set of $f_\alpha$ is $P(f_\alpha)=\partial\Delta_\alpha^0\cup\partial\Delta_\alpha^\infty$, this implies that $\{\gamma_1\}$ forms a curve system in $\EC\setminus P(f_\alpha)$, which is essential and nonperipheral (see \cite[\S B.2]{McM94b}). Note that $f_\alpha:A_\alpha\to\EC\setminus(\overline{\Delta_\alpha^0}\cup\overline{\Delta_\alpha^\infty})$ is a conformal isomorphism. Then $f_\alpha^{-1}(\gamma_1)$ has a connected component $\gamma_2\subset A_\alpha$, which is a Jordan curve lying close to the boundary of $\overline{\Delta_\alpha^0}\cup \overline{U_\alpha^0}\cup \overline{V_\alpha^0}$ and separating $\partial\Delta_\alpha^\infty$ from $\partial\Delta_\alpha^0$.
Since $\gamma_2$ is isotopic to $\gamma_1$ relative to $P(f_\alpha)$, then $\{\gamma_1\}$ determines a transition matrix $A(\gamma_1)$ which is the unit matrix. This implies that the maximal eigenvalue of $A(\gamma_1)$ is $1$. According to \cite[Theorem B.4]{McM94b}, this is impossible.
Therefore, any of $\partial\Delta_{\alpha}^0$ and $\partial{\Delta_{\alpha}^{\infty}}$ cannot contain two critical points.
\end{proof}

\subsection{The Siegel disks move continuously}\label{subsec:continuity}

Let $\D$ be the unit disk. For any $\alpha\in\Sigma_\theta$, there exist two conformal maps
\begin{equation}\label{equ:h-0-infty}
h_{\alpha}^0: \D \to \Delta_{\alpha}^0  \text{\quad and\quad} h_{\alpha}^\infty: \D \to \Delta_{\alpha}^\infty
\end{equation}
such that
\begin{equation}\label{equ:h-condition}
\begin{split}
  (h_{\alpha}^0)^{-1} \circ f^{\circ 2}_{\alpha} \circ h_{\alpha}^0(\zeta) =R_\theta(\zeta),& ~\forall\, \zeta\in \D, \text{ where } h_{\alpha}^0(0)=0,~(h_{\alpha}^0)'(0)>0 \\
  (h_{\alpha}^\infty)^{-1} \circ f^{\circ 2}_{\alpha} \circ h_{\alpha}^\infty(\zeta) =R_\theta(\zeta),&~\forall\, \zeta\in \D, \text{ where }
  h_{\alpha}^\infty(0)=\infty,~(1/h_{\alpha}^\infty)'(0)>0.
\end{split}
\end{equation}
Note that such $h_\alpha^0$ and $h_\alpha^\infty$ are unique.

\begin{prop}\label{continuousity}
Let $\alpha_0\in\Sigma_\theta$ and $(\alpha_n)_{n\geq 1}$ be a sequence in $\Sigma_\theta$ satisfying $\alpha_n \to \alpha_0$ as $n\to\infty$. Then $\partial{\Delta_{\alpha_{n}}^0} \to \partial{\Delta_{\alpha_0}^0}$ and $ \overline{\Delta_{\alpha_{n}}^{0}} \to  \overline{\Delta_{\alpha_0}^{0}}$; meanwhile $\partial{\Delta_{\alpha_{n}}^{\infty}} \to \partial{\Delta_{\alpha_0}^{\infty}}$ and $ \overline{\Delta_{\alpha_n}^{\infty}} \to  \overline{\Delta_{\alpha_0}^{\infty}}$ as $n\to\infty$, with respect to the Hausdorff metric.
\end{prop}

\begin{proof}
It suffices to consider the Siegel disk $\Delta_{\alpha}^0$ for $\alpha\in\Sigma_\theta$ since the argument for $\Delta_{\alpha}^{\infty}$ is similar.
By Lemma \ref{disjoint}, there is a constant $K_0>1$ which is independent of $\alpha$, such that $h_{\alpha}^0:\D \to \Delta_{\alpha}^0$ can be extended to a $K_0$-quasiconformal mapping from $\EC$ to itself, which we still denote by $h_{\alpha}^0:\EC\to\EC$, where $h_{\alpha}^0(\infty)=\infty$.

According to \cite[Theorem 2.4, p.\,17]{Leh87}, there exists a constant $c(K_0)>1$ such that for all $\alpha\in\Sigma_\theta$, we have
\begin{equation}\label{equ:Leh}
\frac{\max_{z\in\partial\D}|h_\alpha^0(z)|}{\min_{z\in\partial\D}|h_\alpha^0(z)|}=\frac{\max\{|z|:z\in\partial\Delta_\alpha^0\}}{\min\{|z|:z\in\partial\Delta_\alpha^0\}}\leq c(K_0).
\end{equation}
Note that the two critical points $\{c_1,c_2\} =\{-(1+\sqrt{1-\lambda})^{-1},~-(1-\sqrt{1-\lambda})^{-1}\}$ of $f_\alpha$ cannot be contained in the Siegel disk $\Delta_\alpha^0$, where $\lambda=e^{2\pi\ii\theta}$. Therefore, by \eqref{equ:Leh} the size of $\Delta_\alpha^0$ cannot be arbitrarily large, i.e., there exists a constant $M=M(K_0)>1$ such that $\overline{\Delta_\alpha^0}\subset\{z:|z|<M\}$ for all $\alpha\in\Sigma_\theta$.

Let $z_0\neq\infty$ be a point satisfying $|z_0|>M$. We assume further that the extended quasiconformal mapping $h_\alpha^0:\EC\to\EC$ satisfies $h_\alpha^0(z_0)=z_0$. Moreover, there exists $\widetilde{K}_0>1$ such that $h_\alpha^0:\EC\to\EC$ is $\widetilde{K}_0$-quasiconformal for all $\alpha\in\Sigma_\theta$. Since $h_\alpha^0$ also fixes $0$ and $\infty$, it implies that $\{h_\alpha^0:\EC\to\EC:\alpha\in\Sigma_\theta\}$ is a normal family (see \cite[Theorem 2.1, p.\,14]{Leh87}).

\medskip

Let $\alpha_n\to\alpha_0\in\Sigma_\theta$ as $n\to\infty$. Passing to a subsequence we assume that $h_{\alpha_n}^0$ converges uniformly to a map $h:\EC\to\EC$ as $n\to\infty$. Since every $h_{\alpha_n}^0$ fixes $0$, $z_0$ and $\infty$, it implies that $h$ is not a constant and hence is a quasiconformal mapping defined from $\EC$ onto itself. Note that every $h_{\alpha_n}^0$ is holomorphic in $\D$ and $(h_{\alpha_n}^0)'(0)>0$, it follows that $h|_{\D}$ is also holomorphic and hence conformal, and $h'(0)>0$. Taking the limit from the both sides of the equation $f_{\alpha_n}^{\circ 2}\circ h_{\alpha_n}^0(\zeta)=h_{\alpha_n}^0\circ R_\theta(\zeta)$ as $n\to\infty$, we obtain $f_{\alpha_0}^{\circ 2}\circ h(\zeta)=h\circ R_\theta(\zeta)$, where $\zeta\in\D$.
Note that $(h_{\alpha_0}^0)^{-1}:\Delta_{\alpha_0}^0\to\D$ is the unique conformal map which conjugates $f_{\alpha_0}^{\circ 2}$ to $R_\theta$, where $h_{\alpha_0}^0(0)=0$ and $(h_{\alpha_0}^0)'(0)>0$. It implies that $h|_{\D}=h_{\alpha_0}^0|_{\D}$. By the continuity of $h$ and $h_{\alpha_0}^0$, we have $h|_{\overline{\D}}=h_{\alpha_0}^0|_{\overline{\D}}$.
Therefore, $h_{\alpha_n}^0(\partial{\D}) \to h_{\alpha_0}^0(\partial{\D})$ and $h_{\alpha_n}^0(\overline{\D}) \to h_{\alpha_0}^0(\overline{\D})$ with respect to the Hausdorff metric as $n\to\infty$.

Let $(\alpha_{n_k})_{k\in\N}$ be any other subsequence of $(\alpha_n)_{n\in\N}$ which satisfies $\alpha_{n_k}\to \alpha_0$ and $h_{\alpha_{n_k}}^0\to \widetilde{h}$ as $k\to\infty$. Similar to the arguments above, $\widetilde{h}|_{\D}$ is conformal and satisfies $f_{\alpha_0}^{\circ 2}\circ \widetilde{h}(\zeta)=\widetilde{h}\circ R_\theta(\zeta)$, where $\widetilde{h}(0)=0$ and $\widetilde{h}'(0)>0$. By the uniqueness of the normalized linearization map, we have $\widetilde{h}=h_{\alpha_0}^0$. This finishes the proof.
\end{proof}

\section{Limit dynamics}\label{limitdynamics}

\subsection{Compactification and the symmetry of $\Sigma_\theta$}

In this subsection, we first study the dynamics of $f_\alpha^{\circ 2}$ as $\alpha$ tends to $\infty$ or $0$.
This allows us to give a compactification $\widehat{\Sigma}_\theta$ of the parameter space $\Sigma_\theta$.
Then we study the symmetry of $\widehat{\Sigma}_\theta$ and the continuous dependence of the Siegel disks of the maps in $\widehat{\Sigma}_\theta$.

\begin{lem}\label{0case}
For any $z\in\EC$, we have
\begin{equation}
\lim_{\alpha\to \infty} f_\alpha^{\circ 2}(z)=g_\infty(z):=\frac{e^{2\pi\ii\theta}(z+z^2)}{1+e^{2\pi\ii\theta} z} \text{\quad and\quad}
\lim_{\alpha\to 0} f_\alpha^{\circ 2}(z)=g_0(z):=\frac{z+z^2}{1+e^{2\pi\ii\theta} z},
\end{equation}
where the first convergence is uniform on any compact subset of $\EC\setminus\{\infty,-1/e^{2\pi\ii\theta}\}$, and the second is uniform on any compact subset of $\EC\setminus\{0,-1\}$.
\end{lem}

\begin{proof}
Denote $\lambda=e^{2\pi\ii\theta}\neq 1$. By a direct calculation, we have
\begin{equation}\label{equ:1}
f_\alpha^{\circ 2}(z)=\frac{\alpha(1+\lambda f_\alpha(z))}{f_\alpha(z)+(f_\alpha(z))^2}=
\frac{z+z^2}{1+\lambda z}\cdot \frac{z+z^2+\alpha\lambda(1+\lambda z)}{z+z^2+\alpha(1+\lambda z)}.
\end{equation}
Therefore, for any $z\in\EC$, we have $f_\alpha^{\circ 2}(z)\to g_\infty(z)$ as $\alpha\to \infty$ and $f_\alpha^{\circ 2}(z)\to g_0(z)$ as $\alpha\to 0$.
If $z+z^2\neq \infty$ and $1+\lambda z\neq 0$, then the convergence $f_\alpha^{\circ 2}\to g_\infty$ is uniform in a neighborhood of $z$ as $\alpha\to \infty$. In particular, the convergence is uniform on any compact subset of $\EC\setminus\{\infty,-1/\lambda\}$. Similarly, the convergence $f_\alpha^{\circ 2}\to g_0$ is uniform on any compact subset of $\EC\setminus\{0,-1\}$ as $\alpha\to 0$.
\end{proof}

By adding the limit maps to $\Sigma_\theta$ at two singularities $\infty$ and $0$, we obtain a compactification of the parameter space $\Sigma_\theta$:
\begin{equation}
\widehat{\Sigma}_\theta:=\Sigma_\theta\cup\{g_\infty,g_0\}.
\end{equation}
Note that the limit maps can be written as
\begin{equation}
g_\infty(z)=\alpha e^{2\pi\ii\theta}/f_\alpha(z) \text{\quad and\quad} g_0(z)=\alpha/f_\alpha(z),
\end{equation}
and they have the following properties:
\begin{itemize}
\item $g_\infty$ has a fixed Siegel disk $\Delta_\infty^0$ centered at $0$ with rotation number $\theta$, and a fixed parabolic basin $U_\infty$ attaching at the $1$-parabolic fixed point $\infty$;
\item $g_0$ has a fixed Siegel disk $\Delta_0^\infty$ centered at $\infty$ with rotation number $\theta$, and a fixed parabolic basin $U_0$ attaching at the $1$-parabolic fixed point $0$; and
\item The critical points of $g_0$ and $g_\infty$ are both $\{c_1,c_2\}$, which is the same to that of $f_\alpha$ for all $\alpha\in\Sigma_\theta$.
\end{itemize}

The parameter space $\Sigma_\theta$ and its compactification $\widehat{\Sigma}_\theta$ have some kind of symmetry, which is shown in the following result and the proof is based on direct calculations:

\begin{lem}\label{lem:sym}
Let $\tau(z)=1/(e^{2\pi\ii\theta}z)$. Then
\begin{enumerate}
\item $\tau\circ f_\alpha\circ \tau^{-1}=f_{\alpha'}$, where $\alpha\in\Sigma_\theta$ and $\alpha'=e^{-6\pi\ii\theta}/\alpha$;
\item $\tau\circ g_0\circ \tau^{-1}=g_\infty$; and
\item $\tau(c_1)=c_2$ and $\tau(c_2)=c_1$.
\end{enumerate}
\end{lem}

By the symmetry of Lemma \ref{lem:sym}, the following result is an immediate consequence of Theorem \ref{zhang11} and Leau-Fatou's flower theorem (\cite[\S 10]{Mil06}):

\begin{cor}\label{cor:marking}
There exists a marking $c_1=c_1(\theta)$, $c_2=c_2(\theta)$ such that
\begin{enumerate}
\item $c_1\in\partial\Delta_\infty^0$ and $c_2\in U_\infty$; and
\item $c_2\in\partial\Delta_0^\infty$ and $c_1\in U_0$.
\end{enumerate}
\end{cor}

The following result shows that the Siegel disks move continuously in the compactified parameter space $\widehat{\Sigma}_\theta$.

\begin{prop}\label{conti-revisited}
We have
\begin{enumerate}
\item If $\alpha\to\infty$, then $\partial{\Delta_{\alpha}^0} \to \partial{\Delta_\infty^0}$, $ \overline{\Delta_{\alpha}^0} \to  \overline{\Delta_\infty^0}$ and $\overline{\Delta_{\alpha}^\infty}\to \{\infty\}$, with respect to the Hausdorff metric; and
\item If $\alpha \to 0$, then $\partial{\Delta_{\alpha}^{\infty}} \to \partial{\Delta_0^{\infty}}$, $ \overline{\Delta_{\alpha}^{\infty}} \to  \overline{\Delta_0^{\infty}}$ and $\overline{\Delta_{\alpha}^0}\to \{0\}$, with respect to the Hausdorff metric.
\end{enumerate}
\end{prop}

\begin{proof}
We only prove (a) since the proof of Part (b) is similar. By Lemma \ref{disjoint}, $\partial\Delta_\infty^0$ is a quasicircle. By Lemma \ref{0case}, the statements $\partial{\Delta_{\alpha}^0} \to \partial{\Delta_\infty^0}$ and $ \overline{\Delta_{\alpha}^0} \to  \overline{\Delta_\infty^0}$ as $\alpha\to\infty$ follow from a completely similar proof of Proposition \ref{continuousity}.

By the continuity of $\partial\Delta_\alpha^0$ as $\alpha\to\infty$, there exist two constants $\varepsilon_0>0$  and $M>1$, such that if $|\alpha|\geq 1/\varepsilon_0$, then
\begin{equation}
\partial\Delta_\alpha^0 \subset\{z\in\C: 1/M<|z|<M\}.
\end{equation}
Note that $f_\alpha(\overline{\Delta_\alpha^\infty})=\overline{\Delta_\alpha^0}$. By the formula of $f_\alpha$ in \eqref{equ:family}, the spherical diameter of $\overline{\Delta_\alpha^\infty}$ tends to $0$ as $\alpha\to \infty$. Since $\infty\in\Delta_\alpha^\infty$, we have $\overline{\Delta_{\alpha}^\infty}\to \{\infty\}$ as $\alpha\to \infty$.
\end{proof}

As an immediate corollary, we have:

\begin{cor}\label{cor:Siegel-crit}
Under the same marking of critical points as in Corollary \ref{cor:marking}, there exists $\delta=\delta(\theta)>1$ such that
\begin{enumerate}
\item If $\delta\leq |\alpha|<+\infty$, then $c_1\in\partial\Delta_\alpha^0$ and $\partial\Delta_\alpha^\infty$ contains no critical point; and
\item If $0<|\alpha|\leq 1/\delta$, then $c_2\in\partial\Delta_\alpha^\infty$ and $\partial\Delta_\alpha^0$ contains no critical point.
\end{enumerate}
\end{cor}

\subsection{Location of the critical points, for specific $\theta$'s}\label{subsec:theta}

Based on Corollary \ref{cor:Siegel-crit}, we know that when $\alpha$ is sufficiently large or small, then $\partial\Delta_\alpha^0\cup\partial\Delta_\alpha^\infty$ contains exactly one critical point. However, we still cannot determine which critical point (This depends on the marking) is contained in the boundary of the Siegel disks. In this subsection, we study this problem for some specific given rotation number $\theta$.

\begin{prop}\label{prop:bounded}
If $\theta\in(\tfrac{1}{6},\tfrac{5}{6})$ is a bounded type irrational number, then under the marking \eqref{equ:c-1-2}, we have
\begin{itemize}
\item If $\delta\leq |\alpha|<+\infty$, then $c_1\in\partial\Delta_{\alpha}^0$; and
\item If $0<|\alpha|\leq 1/\delta$, then $c_2\in \partial\Delta_{\alpha}^\infty$.
\end{itemize}
\end{prop}

\begin{proof}
By Lemma \ref{lem:sym} and Proposition \ref{conti-revisited}, it suffices to prove that for $g_\infty$, the critical point $c_2=-(1-\sqrt{1-\lambda})^{-1}$ is contained in the parabolic basin $U_\infty$, where $\lambda=e^{2\pi\ii\theta}$. A direct calculation shows that
\begin{equation}
g_\infty(z)=\frac{\lambda(z+z^2)}{1+\lambda z}=z+\frac{\lambda-1}{\lambda}-\frac{\lambda-1}{\lambda}\cdot\frac{1}{1+\lambda z}.
\end{equation}
Denote $\eta=(\lambda-1)/\lambda$ and $\varphi(\zeta)=\eta\zeta$. Then
\begin{equation}
\widetilde{g}_\infty(\zeta)=\varphi^{-1}\circ g_\infty\circ\varphi(\zeta)=\zeta+1-\frac{1}{1+(\lambda-1)\zeta}.
\end{equation}
Since $\theta\in(\tfrac{1}{6},\tfrac{5}{6})$, we have $1<|\lambda-1|\leq 2$.
If $\re\zeta>2$, then
\begin{equation}
\re \widetilde{g}_\infty(\zeta)\geq \re\zeta+1-\frac{1}{|\lambda-1|\cdot|\zeta|-1}>\re\zeta.
\end{equation}
This implies that $\{\zeta\in\C:\re\zeta>2\}$ is contained in the immediate parabolic basin of $\widetilde{g}_\infty$.
The corresponding critical point of $\widetilde{g}_\infty$ is $\widetilde{c}_2=\varphi^{-1}(c_2)=c_2/\eta$ and we have
\begin{equation}\label{equ:g-c}
\widetilde{g}_\infty(\widetilde{c}_2)=1+\frac{1}{1-\lambda}+2\sqrt{\frac{1}{1-\lambda}}.
\end{equation}
Note that $\re (\tfrac{1}{1-\lambda})=\tfrac{1}{2}$ for all $\lambda\in\partial\D\setminus\{1\}$. Hence
\begin{equation}
\re\widetilde{g}_\infty(\widetilde{c}_2)>1+\frac{1}{2}+2\cdot\frac{1}{2}>2.
\end{equation}
Therefore, $\widetilde{c}_2$ is contained in the immediate parabolic basin of $\widetilde{g}_\infty$. In particular, $c_2$ is contained in $U_\infty$.
\end{proof}

The result of Proposition \ref{prop:bounded} can be  strengthened if one considers the several iterations $\widetilde{g}_\infty^{\circ n}(\widetilde{c}_2)$ in \eqref{equ:g-c} for some $n\geq 2$. One can see that the two critical points $c_1$ and $c_2$ are arbitrarily close provided $\theta$ is arbitrarily close to $0$. Hence it is hard to distinguish the two critical orbits in this case. But we believe that Proposition \ref{prop:bounded} holds for all bounded type irrational numbers. See Figure \ref{Fig_Julia}.

\begin{figure}[!htpb]
 \setlength{\unitlength}{1mm}
  \centering
  \includegraphics[width=0.95\textwidth]{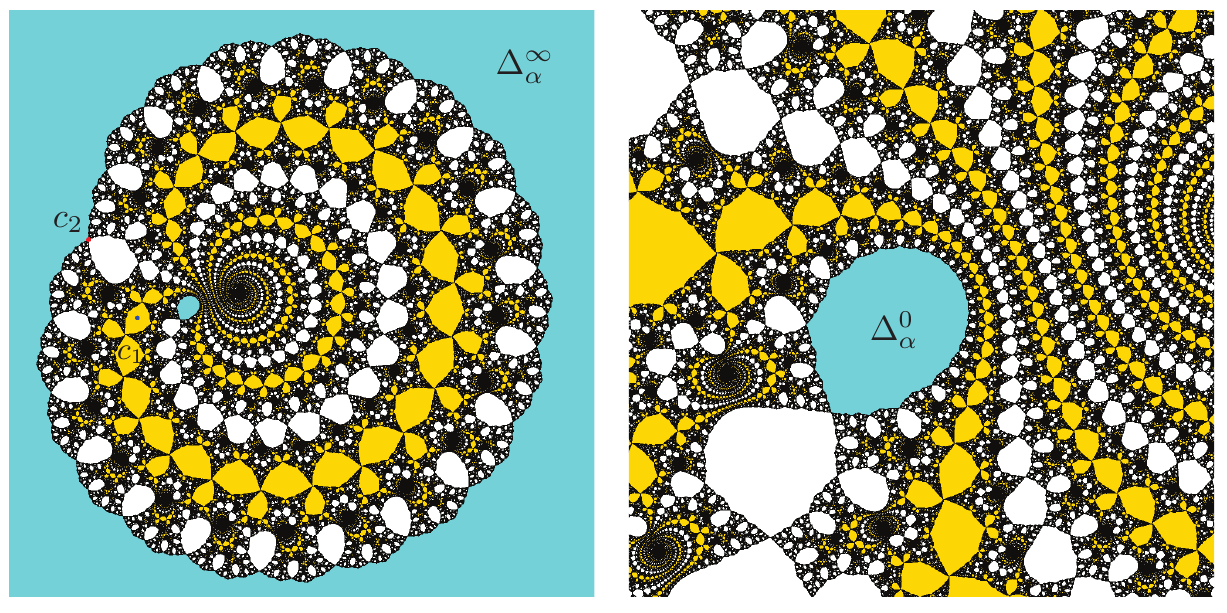}
  \caption{The Julia set of $f_\alpha\in\Sigma_\theta$ and its zoom near $\Delta_\alpha^0$, where $\theta=2/(39+\sqrt{5})=[0;20,1,1,1,\cdots]$
  and $\alpha\approx 0.30689283+ 0.11243024\ii$ are chosen such that $f_\alpha^{\circ 3}(c_1)=c_1$. The critical points $c_1$ and $c_2$ are colored by blue and red respectively. One can observe that $c_2=-(1-\sqrt{1-\lambda})^{-1}\in\partial\Delta_\alpha^\infty$.}
  \label{Fig_Julia}
\end{figure}

In the rest of this paper, without loss of generality, we always assume that the critical points $c_1=c_1(\theta)$ and $c_2=c_2(\theta)$ are marked such that Proposition \ref{prop:bounded} holds, i.e., $c_1\in\partial\Delta_{\alpha}^0$ if $\alpha$ is large enough while $c_2\in \partial\Delta_{\alpha}^\infty$ if $\alpha$ is small enough. Otherwise, one can exchange the subscripts of the critical points by the symmetry of the parameter space as stated in Lemma \ref{lem:sym}.

\section{Dynamics of $f_{\alpha}$ for $\alpha \in \Gamma$}\label{gamma}

Based on Corollary \ref{cor:Siegel-crit}, we consider the following set:
\begin{equation}\label{equ:Gamma}
\Gamma:=\{\alpha\in\Sigma_\theta:\partial\Delta_\alpha^0\cup\partial\Delta_\alpha^\infty \text{ contains } c_1 \text{ and }c_2\}.
\end{equation}
It is easy to see that $\Gamma\neq\emptyset$. Indeed, there exists a unique $\alpha\in\Gamma$ satisfying $f_\alpha(c_1)=c_2$ (or $f_\alpha(c_2)=c_1$).
In this section, we give a combinatorial description of the Fatou set of $f_\alpha$ and prove that the Julia set of $f_\alpha$ has zero two-dimensional Lebesgue measure for all $\alpha\in\Gamma$.

\subsection{The addresses of the Fatou components}\label{subsec:address}

We first prove the following result:

\begin{lem}\label{lem:Gamma-crit}
For any $\alpha\in \Gamma$, we have
\begin{enumerate}
\item $c_1\in\partial\Delta_\alpha^0$ and $c_2\in\partial\Delta_\alpha^\infty$;
\item The Fatou set $F(f_{\alpha})$ of $f_\alpha$ equals to $F(f^{\circ 2}_{\alpha})=\bigcup_{n=0}^{\infty}f^{-2n}_{\alpha}(\Delta_{\alpha}^0 \cup \Delta_{\alpha}^{\infty})$; and
\item $\big(\bigcup_{n=0}^{\infty}f^{-2n}_{\alpha}(\overline{\Delta_{\alpha}^0})\big)  \cap \big(\bigcup_{n=0}^{\infty}f^{-2n}_{\alpha}(\overline{\Delta_{\alpha}^{\infty}})\big)=\emptyset$.
\end{enumerate}
\end{lem}

\begin{proof}
(a) Let $\gamma$ be a simple curve connecting $\delta$ with a given $\alpha\in\Gamma$, where $\delta>1$ is the number introduced in Corollary \ref{cor:Siegel-crit}. By the assumption at the end of last section, we have $c_1\in\partial\Delta_\delta^0$. According to Proposition \ref{continuousity}, the $2$-cycle of Siegel disks move continuously as the parameter varies on $\gamma$. By Lemma \ref{disjoint}, we have $c_1\in\partial\Delta_\alpha^0$ and hence $c_2\in\partial\Delta_\alpha^\infty$.

\medskip
(b) Since both critical points are contained in boundaries of the cycle $\{\Delta_\alpha^0,\Delta_\alpha^\infty\}$, it implies that all Fatou components of $f_\alpha$ are iterated eventually onto this cycle of Siegel disks. The statement follows since $f_\alpha^{\circ 2}(\Delta_\alpha^0)=\Delta_\alpha^0$ and $f_\alpha^{\circ 2}(\Delta_\alpha^\infty)=\Delta_\alpha^\infty$.

\medskip
(c) If $f_\alpha^{-2m}(\overline{\Delta_\alpha^0})\cap f_\alpha^{-2n}(\overline{\Delta_\alpha^\infty})\neq\emptyset$ for some $m$, $n\geq 0$, then $\overline{\Delta_\alpha^0}\cap \overline{\Delta_\alpha^\infty}\neq\emptyset$. This contradicts Lemma \ref{disjoint}.
\end{proof}

In the rest of this subsection, we give a description of the combination of the Fatou components which are iterated eventually onto $\Delta_{\alpha}^0$ (resp. $\Delta_{\alpha}^\infty$) under $f^{\circ 2}_{\alpha}$. Specifically, we shall give an address to every component of $\bigcup_{n=0}^{\infty}f^{-2n}_{\alpha}(\Delta_{\alpha}^0\cup \Delta_{\alpha}^\infty)$ by following \cite[\S 0]{Pet96}.

\medskip
Since $f_\alpha^{\circ 2}$ is a rational map of degree $4$, it has $6$ critical points (counted with multiplicity): $c_1$, $c_1'$, $c_1''$ and $c_2$, $c_2'$, $c_2''$, where
\begin{equation}\label{equ:crit}
f_\alpha(c_1')=f_\alpha(c_1'')=c_2 \text{\quad and\quad} f_\alpha(c_2')=f_\alpha(c_2'')=c_1.
\end{equation}
Moreover, $c_1$, $c_1'\in\partial\Delta_\alpha^0$ and $c_2$, $c_2'\in\partial\Delta_\alpha^\infty$. Since $c_1$ and $c_2$ cannot be periodic, we have
\begin{equation}\label{equ:f-c1}
f_\alpha^{\circ n}(c_1)\neq c_2' \text{\quad and\quad} f_\alpha^{\circ n}(c_2)\neq c_1', \text{\quad for any } n\geq 0.
\end{equation}
It is easy to see that there is a Fatou component in $\EC\setminus(\overline{\Delta_\alpha^0}\cup\overline{\Delta_\alpha^\infty})$ attaching at $z\in\partial\Delta_\alpha^0\cup\partial\Delta_\alpha^\infty$ if and only if
\begin{equation}
\begin{split}
z\in\Upsilon_\alpha:=&~\bigcup_{n\geq 0}f_\alpha^{-n}(\{c_1,c_2\})\cap (\partial\Delta_\alpha^0\cup\partial\Delta_\alpha^\infty) \\
=&~\Big(\bigcup_{n\geq 0}f_\alpha^{-2n}(\{c_1,c_1'\})\cap \partial\Delta_\alpha^0\Big) \cup
\Big(\bigcup_{n\geq 0}f_\alpha^{-2n}(\{c_2,c_2'\})\cap \partial\Delta_\alpha^\infty\Big).
\end{split}
\end{equation}
For $z\in\partial\Delta_\alpha^0$, we define
\begin{equation}
\chi(z):=\sharp\,\{U:\,\overline{U}\cap \overline{\Delta_\alpha^0}=\{z\}, \text{ where } U \text{ is a Fatou component of } f_\alpha\}.
\end{equation}
Similarly, one can define $\chi(z)$ for all $z\in\partial\Delta_\alpha^\infty$. Therefore, $\chi(z)\geq 1$ if and only if $z\in\Upsilon_\alpha$.

\begin{lem}\label{lem:confir}
Let $\alpha\in\Gamma$. Then one of the following three cases happens:
\begin{enumerate}
\item $f_\alpha^{\circ (2m+1)}(c_1)=c_2$ for some $m\geq 0$ and $f_\alpha^{\circ (2n+1)}(c_2)\neq c_1$ for any $n\geq 0$. Then $\chi(z)=1$ if $z\in\big(\bigcup_{k=0}^{2m}f_\alpha^{-k}(c_2)\big)\cap\Upsilon_\alpha$ and $\chi(z)=3$ if $z\in\bigcup_{k\geq 0}f_\alpha^{-k}(c_1)=\Upsilon_\alpha\setminus \bigcup_{k=0}^{2m}f_\alpha^{-k}(c_2)$;
\item $f_\alpha^{\circ (2m+1)}(c_2)=c_1$ for some $m\geq 0$ and $f_\alpha^{\circ (2n+1)}(c_1)\neq c_2$ for any $n\geq 0$. Then $\chi(z)=1$ if $z\in\big(\bigcup_{k=0}^{2m}f_\alpha^{-k}(c_1)\big)\cap\Upsilon_\alpha$ and $\chi(z)=3$ if $z\in\bigcup_{k\geq 0}f_\alpha^{-k}(c_2)=\Upsilon_\alpha\setminus \bigcup_{k=0}^{2m}f_\alpha^{-k}(c_1)$;
\item $f_\alpha^{\circ (2n+1)}(c_1)\neq c_2$ and $f_\alpha^{\circ (2n+1)}(c_2)\neq c_1$ for any $n\geq 0$. Then $\chi(z)=1$ for all $z\in\Upsilon_\alpha$.
\end{enumerate}
\end{lem}

\begin{proof}
(a) If $f_\alpha^{\circ (2m+1)}(c_1)=c_2$ for some $m\geq 0$, then $f_\alpha^{\circ (2n+1)}(c_2)$ cannot be $c_1$ for any $n\geq 0$ since otherwise, $c_1$ and $c_2$ would be periodic, which is impossible. Moreover, we have $f_\alpha^{\circ n}(c_1)\neq c_1$ and $f_\alpha^{\circ n}(c_2)\neq c_2$ for any $n\geq 1$. This implies that $\chi(c_2)=1$. Since the local degree of $f_\alpha^{\circ (2m+1)}$ at $c_1$ is $2$, we have $\chi\big(f_\alpha^{\circ k}(c_1)\big)=1$ for $1\leq k\leq 2m+1$ and $\chi(c_1)=3$. The rest statement follows immediately.

\medskip
(b) The proof is completely similar to (a).

\medskip
(c) If $f_\alpha^{\circ (2n+1)}(c_1)\neq c_2$ and $f_\alpha^{\circ (2n+1)}(c_2)\neq c_1$ for any $n\geq 0$, then $\chi(c_1)=\chi(c_2)=1$. This implies that $\chi(z)=1$ for all $z\in\Upsilon_\alpha$.
\end{proof}

In the following, we focus our attention on the combination of the components of $\bigcup_{n\geq 0}f_\alpha^{-2n}(\Delta_{\alpha}^0)$. The configuration of the components of $\bigcup_{n\geq 0}f_\alpha^{-2n}(\Delta_{\alpha}^\infty)$ can be described in the same way.
Note that $f^{-2}_{\alpha}(\Delta_{\alpha}^0)$ consists of $4$ components: $\Delta_{\alpha}^0$ and $3$ other components $U_0^0$, $U_0^1$ and $U_0^2$, where $f_\alpha(U_0^0)=\Delta_\alpha^\infty$, $U_0^0$ and $U_0^1$ are attached to $\partial\Delta_{\alpha}^0$ at $c_1$ and $c_1'$ respectively, and $U_0^2$ is attached to $\partial U_0^0$ at $c_1''$.
Based on Lemma \ref{lem:confir}, there are following $3$ cases (see Figure \ref{Fig:map}):
\begin{itemize}
\item[(i)] If $f_\alpha(c_1)=c_2$, then $c_1=c_1'=c_1''$ and $\chi(c_1)=3$;
\item[(ii)] If $f_\alpha^{\circ (2m+1)}(c_1)=c_2$ for some $m\geq 1$, then $c_1$, $c_1'$, $c_1''$ are pairwise different, $\chi(c_1)=3$ and $\chi(c_1')=1$;
\item[(iii)] If $f_\alpha^{\circ (2n+1)}(c_1)\neq c_2$ for any $n\geq 0$, then $c_1$, $c_1'$, $c_1''$ are pairwise different and $\chi(c_1)=\chi(c_1')=1$.
\end{itemize}

\begin{figure}[htbp]
  \setlength{\unitlength}{1mm}
    \centering
  \includegraphics[width=0.92\textwidth]{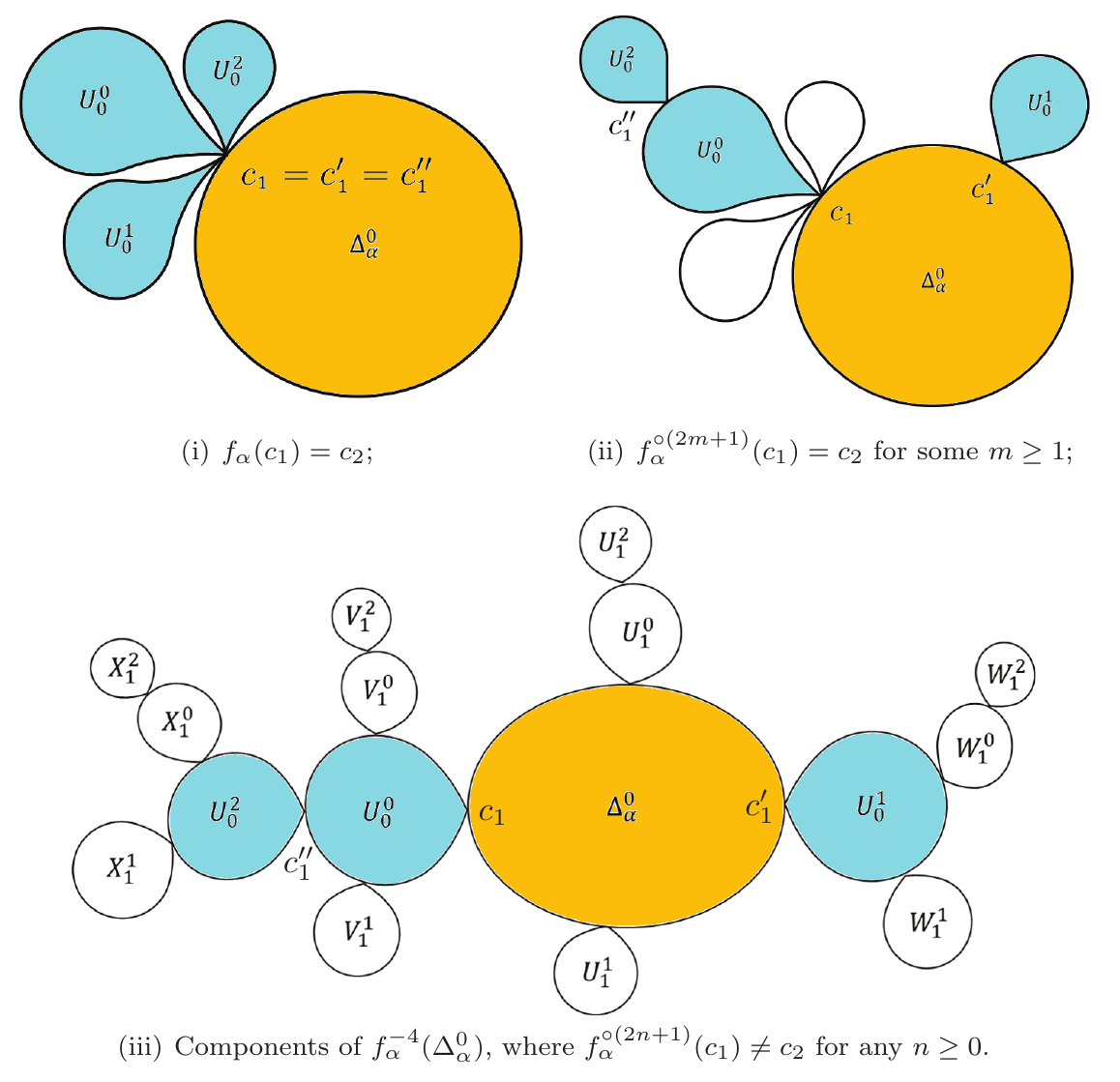}
  \caption{The locations of some preimages of $\Delta_{\alpha}^0$ under $f_\alpha^{\circ 2}$.}
  \label{Fig:map}
\end{figure}

In the following, we only give a detailed description of the locations of the components of $\bigcup_{n\geq 0}f_\alpha^{-2n}(\Delta_{\alpha}^0)$ of Case (iii), i.e., the forward orbit of $c_1$ does not touch $c_2$. The rest two cases can be handled completely similarly. For $t\in\{0,1,2\}$, $f_\alpha^{-2}(U_0^t)$ consists of $4$ components $U_{1}^t$, $V_{1}^t$, $W_{1}^t$ and $X_{1}^t$ (see Figure \ref{Fig:map}), where
\begin{itemize}
\item $\{U_1^0,U_1^1\}$, $\{V_1^0,V_1^1\}$, $\{W_1^0,W_1^1\}$, $\{X_1^0,X_1^1\}$ are attached to $\Delta_\alpha^0$, $U_0^0$, $U_0^1$ and $U_0^2$ respectively; and
\item $U_1^2$, $V_1^2$, $W_1^2$, $X_1^2$ are attached to $U_1^0$, $V_1^0$, $W_1^0$, $X_1^0$ respectively.
\end{itemize}
For integers $s\geq 1$ and $t\in\{0,1,2\}$, let $U_s^t$, $V_s^t$, $W_s^t$ and $X_s^t$ be the unique components of $f_\alpha^{-2s}(U_0^t)$ such that
\begin{itemize}
\item $\{U_s^0,U_s^1\}$, $\{V_s^0,V_s^1\}$, $\{W_s^0,W_s^1\}$, $\{X_s^0,X_s^1\}$ are attached to $\Delta_\alpha^0$, $U_0^0$, $U_0^1$ and $U_0^2$ respectively; and
\item $U_s^2$, $V_s^2$, $W_s^2$, $X_s^2$ are attached to $U_s^0$, $V_s^0$, $W_s^0$, $X_s^0$ respectively.
\end{itemize}

Inductively, for any $m\geq 2$ and sequences $(s_1,\cdots, s_m)$, $(t_1,\cdots,t_m)$, where $s_i\geq 1$ and $t_i\in\{0,1,2\}$ with $1\leq i\leq m$, let $U_{s_1,\cdots,s_m}^{t_1,\cdots,t_m}$, $V_{s_1,\cdots,s_m}^{t_1,\cdots,t_m}$, $W_{s_1,\cdots,s_m}^{t_1,\cdots,t_m}$ and $X_{s_1,\cdots,s_m}^{t_1,\cdots,t_m}$ be the unique components of $f_\alpha^{-2(s_1+\cdots +s_m+1)}(\Delta_\alpha^0)$ satisfying
\begin{itemize}
\item $\{U_{s_1,\cdots,s_{m-1},s_m}^{t_1,\cdots,t_{m-1},0},U_{s_1,\cdots,s_{m-1},s_m}^{t_1,\cdots,t_{m-1},1}\}$ are attached to $U_{s_1,\cdots,s_{m-1}}^{t_1,\cdots,t_{m-1}}$;
\item $U_{s_1,\cdots,s_{m-1},s_m}^{t_1,\cdots,t_{m-1},2}$ is attached to $U_{s_1,\cdots,s_{m-1},s_m}^{t_1,\cdots,t_{m-1},0}$;
\item The above relations hold similarly for $V_{s_1,\cdots,s_m}^{t_1,\cdots,t_m}$, $W_{s_1,\cdots,s_m}^{t_1,\cdots,t_m}$ and $X_{s_1,\cdots,s_m}^{t_1,\cdots,t_m}$;
\item There are following mapping relations:
\begin{equation}
f_\alpha^{\circ 2}(Y_{s_1,s_2,\cdots,s_m}^{t_1,t_2,\cdots,t_m})=
\left\{
\begin{array}{ll}
U_{s_1-1,s_2\cdots,s_m}^{t_1,t_2,\cdots,t_m},  &~~~~~~~\text{if}~s_1\geq 2, \\
V_{s_2,\cdots,s_m}^{t_2,\cdots,t_m},  &~~~~~~~\text{if}~s_1=1 \text{ and } t_1=0, \\
W_{s_2,\cdots,s_m}^{t_2,\cdots,t_m},  &~~~~~~~\text{if}~s_1=1 \text{ and } t_1=1, \\
X_{s_2,\cdots,s_m}^{t_2,\cdots,t_m},  &~~~~~~~\text{if}~s_1=1 \text{ and } t_1=2,
\end{array}
\right.
\end{equation}
where $Y=U$, $V$, $W$ or $X$.
\end{itemize}

Based on the above setting, every connected component of $\bigcup_{n\geq 0}f_\alpha^{-2n}(\Delta_{\alpha}^0)$ which is different from $\Delta_\alpha^0$, $U_0^0$, $U_0^1$ and $U_0^2$, corresponds to a unique address $\{(s_1,\cdots,s_m),(t_1,\cdots,t_m)\}$ (together with the marking $U$, $V$, $W$ or $X$), where $s_i\geq 1$ and $t_i\in\{0,1,2\}$. Moreover, the relative positions of these components are uniquely determined by the rotation number $\theta$. Such kind of addresses can be marked also for the components of $\bigcup_{n\geq 0}f_\alpha^{-2n}(\Delta_{\alpha}^\infty)$ similarly. The locations of the Fatou components of $f_\alpha$ are uniquely determined by these addresses and the rotation number $\theta$. We shall use these information to study the rigidity of $f_\alpha$ when $\alpha$ moves in $\Gamma$.

\subsection{Lebesgue measure of the Julia sets}

The main aim in this subsection is to prove the following result:

\begin{lem}\label{Juliaset}
For each $\alpha \in \Gamma$, the Julia set of $f_{\alpha}$ has zero Lebesgue measure.
\end{lem}

To prove Lemma \ref{Juliaset}, we shall prove that the set of points whose forward orbits under $f_\alpha^{\circ 2}$ tend to the post-critical set
\begin{equation}
P(f_\alpha^{\circ 2})=P(f_\alpha)=\partial\Delta_\alpha^0\cup\partial\Delta_\alpha^\infty
\end{equation}
has zero Lebesgue measure. Since for almost all points in the Julia set of $f_\alpha^{\circ 2}$, the forward orbits tend to the post-critical set, we conclude that the whole Julia set has zero Lebesgue measure.

For $f_\alpha^{\circ 2}$, we first provide a quasi-Blaschke product model $F$, such that the unit disk $\D$ corresponds to $\Delta_\alpha^0$ (a similar model can be established such that $\D$ corresponds to $\Delta_\alpha^\infty$). For any point $z$ whose forward orbit tending to $\partial\D$ by $F$, we pull back the geometry near a critical point on $\partial\D$ to a neighborhood of $z$ by a bounded distortion. This implies that $z$ is a not a Lebesgue density point of the points whose forward orbits tending to $\partial\D$. The main idea of the proof of Lemma \ref{Juliaset} is inspired by \cite{McM98b} and \cite{Zha08c}.

\medskip
For rational maps, the typical behaviors of the forward orbits of the points are characterized in the following result. See \cite{Lyu83b} and also \cite[\S 3.3]{McM94b}.
\begin{lem}\label{lem:typical}
Let $f$ be a rational map of degree at least two. Then either
\begin{itemize}
 \item The Julia set $J(f)$ is equal to the whole Riemann sphere; or
 \item The spherical distance between $f^{\circ k}(z)$ and the post-critical set $P(f)$ tends to $0$ for almost every $z\in J(f)$ as $k\to \infty$.
\end{itemize}
\end{lem}

In the following we assume that $f_\alpha^{\circ 2}$ has exactly two different critical points $c_1$ and $c_1'$ on $\partial\Delta_\alpha^0$. The case that $c_1=c_1'$ can be treated completely similarly (This case happens only when $f_\alpha(c_1)=c_2$, see \S\ref{subsec:address}). We now define the model map $F$.
Let $Z_0^0$ be a quasidisk in $\C\setminus\overline{\D}$ which is attached to $\partial\D$ at $c=1$ such that $\partial Z_0^0$ is smooth in a neighborhood of $1$ except at $1$ itself and the two angles formed by $\partial Z_0^0$ and $\partial\D$ are both $\frac{\pi}{3}$. Let $\psi: \EC\setminus (\Delta^{0}_{\alpha}\cup U_0^0)\to\EC\setminus (\D\cup Z_0^0)$ be a continuous map such that
\begin{equation}
\psi: \EC\setminus (\overline{\Delta^{0}_{\alpha}\cup U_0^0})\to\EC\setminus (\overline{\D\cup Z_0^0})
\end{equation}
is conformal and satifies $\psi(\infty)=\infty$, $\psi(c_1)=1$, $\psi(\partial\Delta_\alpha^0)=\partial\D$ and $\psi(\partial U_0^0)=\partial Z_0^0$.
Since $\partial\Delta_\alpha^0$ and $\partial Z_0^0$ are both quasicircles, $\psi$ can be extend to a quasiconformal mapping $\psi:\EC\to\EC$ such that $\psi(\Delta_\alpha^0)=\D$ and $\psi(U_0^0)=Z_0^0$.

For $z\in\EC$, we use $z^*=1/\overline{z}$ to denote the symmetric image of $z$ about the unit circle. Let $Z^*=\{z^*:z\in Z\}$, where
$$Z:=\psi(U_0^0\cup U_0^1\cup U_0^2).$$
We define the model map
\begin{equation}
F(z):=\left\{
\begin{array}{ll}
\psi\circ f_\alpha^{\circ 2}\circ \psi^{-1} (z) & \text{\quad if } |z|\ge 1,\\
(\psi\circ f_\alpha^{\circ 2}\circ \psi^{-1} (z^*))^* &\text{\quad if } |z|< 1.
\end{array}
\right.
\end{equation}
Then $F:\EC\to\EC$ is a quasiregular map and $F: \EC\setminus(Z\cup Z^*) \to \EC$ is holomorphic. Moreover, $F$ has two double critical points $1$ and $\psi(c_1')$ on $\partial\D$.

\medskip
For $z\in\C$, we use $\D_r(z)$ to denote the Euclidean disk centered at $z$ with radius $r>0$. Take $0<\varepsilon<1/12$. Let $\ell_1$ and $\ell_1'$ be the two rays starting from the critical point $c=1$ of $F$ such that the angles between $\partial\D$ and $\ell_1$, $\partial\D$ and $\ell_1'$ are both equal to $\varepsilon \pi$. Let $S_{\varepsilon}^{c}$ be the open cone spanned by $\ell_1$ and $\ell_1'$ which attaches at $c$ from the outside of $\D$.

Note that $\psi:\EC\to\EC$ is a quasiconformal homeomorphism satisfying $\psi(\partial\Delta_\alpha^0)=\partial\D$ and $\psi(c_1)=1$. This implies that the boundaries of $\psi(U_0^1)$ and $\psi(U_0^2)$ are quasicircles. Define
\begin{equation}
\Omega_{\varepsilon,r}^c:= S_{\varepsilon}^{c}\cap \D_r(c) \cap (\C\setminus(\overline{Z\cup \D}))
\text{\quad and\quad}
\Omega:=\C\setminus \big(\overline{\D}\cup\psi(\overline{\Delta_\alpha^\infty})\big).
\end{equation}
Note that the two angles formed by $\partial Z_0^0$ and $\partial\D$ are both $\frac{\pi}{3}$. There exists a small $r>0$ such that $\D_r(c)\cap \psi(U_0^1\cup U_0^2)=\emptyset$ and $\Omega_{\varepsilon,r}^c$ consists of two simply connected domains. In the following, we always assume that $r>0$ is small such that the above properties are satisfied.
See the picture in Figure \ref{contract} on the left.

\begin{figure}[htbp]
  \setlength{\unitlength}{1mm}
  \centering
  \includegraphics[width=0.96\textwidth]{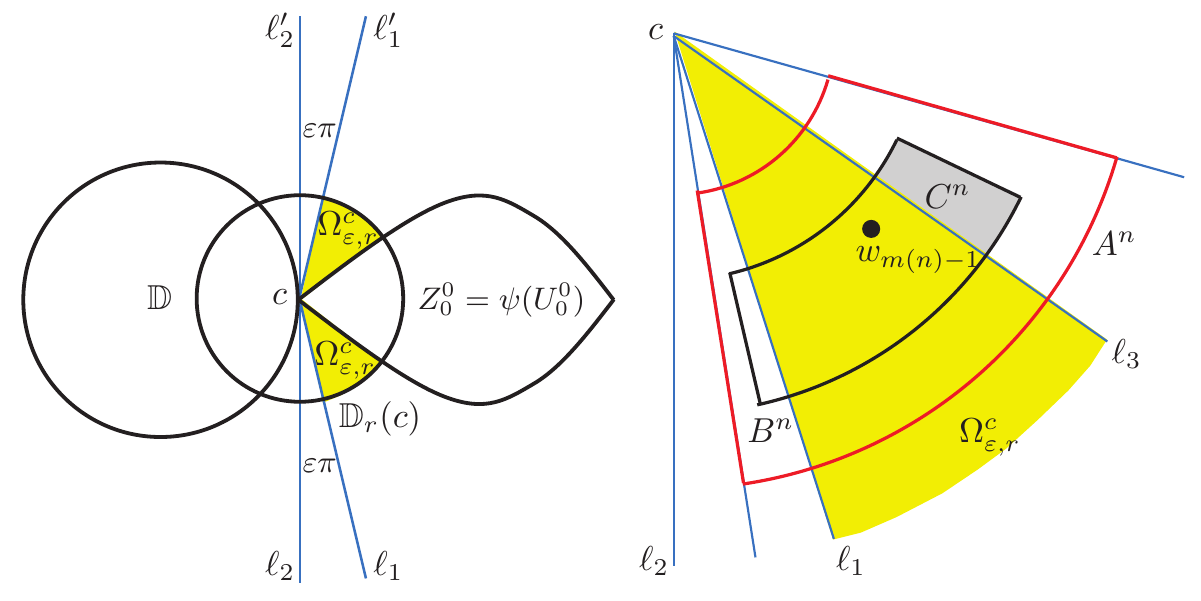}
  \caption{Left: The contraction region $\Omega_{\varepsilon,r}^c$ of $F^{-1}$. Right: Construction of the sequence $(A^n,B^n,C^n,w_{m(n)-1})$.}
  \label{contract}
\end{figure}

Let $\rho_{\Omega}(z)|dz|$ be the hyperbolic metric on $\Omega$. By a similar proof to \cite[Lemma 3.2]{Zha08c}, for the given small $r>0$, there exists $\nu=\nu(\varepsilon,r)>0$ such that for any $z\in \Omega_{\varepsilon,r}^c$, we have
\begin{equation}\label{contracts}
\rho_\Omega(F(z))|F'(z)|\ge(1+\nu)\rho_\Omega(z).
\end{equation}

\begin{proof}[{Proof of Lemma \ref{Juliaset}}]
We only consider the case that $f_\alpha^{\circ 2}$ has exactly two different critical points $c_1$ and $c_1'$ on $\partial\Delta_\alpha^0$. The case $c_1=c_1'$ corresponds to $f_\alpha(c_1)=c_2$, which can be proved completely similarly.
To prove that $J(f_\alpha)$ has zero Lebesgue measure, by Lemma \ref{lem:typical}, it suffices to prove that
\begin{equation}
\Pi:=\{z\in J(f_\alpha): f_\alpha^{\circ n}(z) \text{ tends to } P(f_\alpha) \text{ as } n\to\infty\}=\Pi_0\cup\Pi_\infty
\end{equation}
has zero Lebesgue measure, where
\begin{equation}
\begin{split}
\Pi_0:=&~\{z\in J(f_\alpha^{\circ 2}): f_\alpha^{\circ (2n)}(z) \text{ tends to } \partial\Delta_\alpha^0 \text{ as } n\to\infty\}, \text{ and} \\
\Pi_\infty:=&~\{z\in J(f_\alpha^{\circ 2}): f_\alpha^{\circ (2n)}(z) \text{ tends to } \partial\Delta_\alpha^\infty \text{ as } n\to\infty\}.
\end{split}
\end{equation}
In the following we prove that $\Pi_0$ has zero Lebesgue measure. The case for $\Pi_\infty$ is completely similar. Since the restriction of the quasiconformal map $\psi:\EC\to\EC$ on the Julia set of $f_\alpha^{\circ 2}$ is a conjugacy between $f_\alpha^{\circ 2}$ and $F$, it suffices to prove that the following set has zero Lebesgue measure:
\begin{equation}
\Pi_0':=\{w\in \EC\setminus\bigcup_{k\geq 0}F^{-k}(\overline{\D}): F^{\circ n}(w) \text{ tends to } \partial\D \text{ as } n\to\infty\}.
\end{equation}

Let $w_0\in\Pi_0'$ and denote $w_n=F^{\circ n}(w_0)$. By a completely similar argument as \cite[Lemma 4.11]{Zha08c} (see also \cite[Lemma 6.1]{Zha22}), for the given small $\varepsilon>0$ and $r>0$, there exists a large integer $n_0\geq 1$, such that for any $n\geq n_0$, there exists an integer $1\leq m(n)\le n$ satisfying
\begin{equation}\label{w}
w_{m(n)-1}\in \Omega_{\varepsilon,r}^c,
\end{equation}
where $m(n)$ is an increasing sequence tending to $\infty$ as $n\to\infty$.

For small $r>0$, there are two domains contained in $\D_r(c)\setminus\overline{\D}$ which are tangent to $\partial\D$ at $c$ and are mapped by $F$ into the outside of $\overline{\D}$. We may assume $w_{m(n)-1}$ lies in one of them, say $U$, whose boundary contains an arc of $\partial\D$ in the lower half plane. Meanwhile there is exactly one domain $V\subset \D_r(c)\setminus\overline{\D}$, which is attached to $\partial\D$ at $c$ and mapped by $F$ into the inside of $\D$. Let $\ell_2$ and $\ell_3$ be the two half rays which are tangent with $U$ at $c$. When viewed from $w_{m(n)-1}$, $U$ is approximately an angle domain with boundary $\ell_2$ and $\ell_3$. It follows that the angle between $\ell_2$ and $\ell_3$ is $\pi/3$. Recall that $\ell_1$ is the straight segment between $\ell_2$ and $\ell_3$ which coincides with a part of the boundary of $\Omega_{\varepsilon,r}^c$. The angle between $\ell_1$ and $\ell_2$ is $\varepsilon \pi$.

We consider the polar coordinate system formed by $(c,\ell_2)$. Since $w_{m(n)-1}\in\Omega_{\varepsilon,r}^c$, we have $w_{m(n)-1}=r_0 e^{\beta_0\pi\ii}$ for some $\varepsilon<\beta_0<1/3$ and $0<r_0<r$. We define the following $3$ simply connected domains:
\begin{equation}
\begin{split}
   A^n = &~ \{r e^{\beta\pi\ii}:\tfrac{r_0}{2}<r<\tfrac{3r_0}{2} \text{ and } \tfrac{\varepsilon}{2}<\beta<\tfrac{1}{3}+\varepsilon\}, \\
   B^n = &~ \{r e^{\beta\pi\ii}:\tfrac{3r_0}{4}<r<\tfrac{5r_0}{4} \text{ and } \tfrac{3\varepsilon}{4}<\beta<\tfrac{1}{3}+\tfrac{\varepsilon}{2}\}, \text{ and} \\
   C^n = &~ B^n \cap V= B^n\cap \psi(U_0^0).
\end{split}
\end{equation}
This implies that there exist constants $M_1$, $M_2$, $M_3>0$ which are independent of $n$, such that for every $n\geq n_0$, we have $C^{n} \subset B^{n} \Subset A^{n} \Subset \Omega$ and
\begin{itemize}
  \item $\omega_{m(n)-1}\in B^{n}$ and $F(C^{n}) \subset \D$; and
  \item $\textup{mod}(A^{n}\setminus B^{n})\ge M_1$, $\area(C^{n})/\diam(B^{n})^2 \ge M_2$ and $\diam_{\Omega}(A^{n})\le M_3$.
\end{itemize}
See the picture in Figure \ref{contract} on the right.

\medskip
We consider the pull backs of $(A^{n},B^{n},C^{n},w_{m(n)-1})$ along the orbit $\{w_l:0\leq l\leq m(n)-1\}$, where $n>n_0$. For $0\le l\le m(n)-1$, let $A_l^n$ be the connected component of $F^{-(m(n)-1-l)}(A^{n})$ containing $w_l$. We use $B_l^{n}$ and $C_l^{n}$ to denote the pull backs of $B^{n}$ and $C^{n}$ by $F^{\circ (m(n)-1-l)}$ in $A^{n}_l$. Then $C^{n}_l\subset B^{n}_l\Subset A^{n}_l\Subset\Omega$.
In particular, $A^{n}_0$ is the connected component of $F^{-(m(n)-1)}(A^{n})$ containing $w_0$, and $A^{n}_{m(k)-1}$ is the connected component of $F^{-(m(n)-m(k))}(A^{n})$ containing $w_{m(k)-1}$ for $n_0\le k<n$.

By Schwarz-Pick's lemma, for all $0\le l\le m(n)-1$, we have $\diam_{\Omega}(A^n_l)\le M_3$.
Since $w_n$ tends to $\partial{\D}$ and $m(n)\to \infty$ as $n\to\infty$, there exist $n_1\geq n_0$ and a constant $0<\mu<1$ such that for all $k\ge n_1$,
\begin{equation}\label{19}
w_{m(k)-1}\in A^n_{m(k)-1} \subset \Omega_{\mu\varepsilon,r}^c.
\end{equation}
By (\ref{contracts}) and (\ref{19}), it follows that there is a positive number $\widetilde{\nu}=\widetilde{\nu}(\mu\varepsilon,r)<1$ which is independent of $n$ such that for every $k$ with $n_1\le k\le n$, we have
\begin{equation}\label{20}
\diam_{\Omega}(A^n_{m(k)-1})\le (1-\widetilde{\nu})\,\diam_{\Omega}(A^n_{m(k)}).
\end{equation}
Since $\{m(k):k\geq n_1\}$ is an infinite sequence, it follows that $\diam(A^n_0)\to 0$  as $n\to \infty$. Hence $\diam(B^n_0)\to 0$ as $n\to \infty$.

\medskip
Note that $F^{\circ (m(n)-1)}: A_0^n\to A^n$ is conformal. By Koebe's distortion theorem, we obtain a constant $0<C<\infty$ such that for all $n\geq n_1+1$, the following properties hold:
\begin{itemize}
 \item $w_0\in B_0^n$, $C^{n}_0 \subset B_0^n$ and $F^{\circ m(n)}(C_0^n)\subset\D$; and
 \item $\area(C_0^n)\ge C \diam (B_0^n)^2$.
\end{itemize}
This implies that $w_0$ is not a Lebesgue density point of $\Pi_0'$. By the arbitrariness of $w_0$, $\Pi_0'$ has zero Lebesgue measure and the proof is complete.
\end{proof}

\section{Proof of the Main Theorem}\label{proof}

Let $\Gamma$ be the set defined in \eqref{equ:Gamma}.
In this section we first define a map $A(\alpha)$ on the set $\Gamma$, which measures the conformal angle between the two critical points on the boundaries of the Siegel disks. Then we prove that $A$ is a homeomorphism between $\Gamma$ and a circle. The Main Theorem then follows immediately.

\medskip
To find points on $\Gamma$, let $\eta: (0,1) \to \Sigma_\theta$ be any simple (hence continuous) curve satisfying
\begin{equation}
\lim_{t\to 0^+} \eta(t)=0 \text{\quad and\quad}\lim_{t\to 1^-} \eta(t)=\infty.
\end{equation}
Define
$$t_0=t_0(\eta):=\sup\{0<t<1:\,c_2\in\partial{\Delta_{\eta(t)}^\infty} \text{ and } c_1\not\in\partial{\Delta_{\eta(t)}^0} \}.$$
As mentioned at the end of \S\ref{limitdynamics},  $c_2\in \partial\Delta_{\alpha}^\infty$ and $c_1\not\in \partial\Delta_{\alpha}^0$ if $\alpha$ is small enough, and $c_1\in\partial\Delta_{\alpha}^0$ and $c_2\not\in\partial\Delta_{\alpha}^\infty$ if $\alpha$ is large enough.
By Corollary \ref{cor:Siegel-crit}, we have $0<t_0<1$ and $1/\delta\leq\eta(t_0)\leq \delta$ for a constant $\delta=\delta(\theta)>1$.

\begin{lem}\label{lem:alpha-0}
We have $\alpha_0:=\eta(t_0)\in\Gamma$.
\end{lem}

\begin{proof}
By Proposition \ref{continuousity}, we have
$$d_{H}(\partial{\Delta^{0}_{\eta(t)}}, \partial{\Delta^{0}_{\alpha_0}}) \to 0 \text{\quad and\quad} d_{H}(\partial{\Delta^{\infty}_{\eta(t)}}, \partial{\Delta^{\infty}_{\alpha_0}}) \to 0$$
as $t \to t_0$, where $d_H(\cdot,\cdot)$ denotes the Hausdorff distance in $\C$. By the definition of $t_0$, there is a sequence $t_n\to t_0^-$ such that $\partial{\Delta^{\infty}_{\eta(t_n)}}$ passes through $c_2$ for every $n\ge 1$. Thus $c_2 \in \partial{\Delta^{\infty}_{\alpha_0}}$.

If $c_1\not\in\partial\Delta_{\alpha_0}^{0}$, then by the continuous dependence of the boundaries of the Siegel disks, there exists a small neighborhood $W$ of $\alpha_0$ such that $c_1\not\in\partial\Delta_\alpha^{0}\cup \partial\Delta_\alpha^\infty$ for any $\alpha\in W$. This implies that $c_2\in\partial\Delta_\alpha^\infty$ and $c_1\not\in\partial\Delta_\alpha^{0}$ for all $\alpha\in W$, which contradicts the definition of $\alpha_0$. Therefore, $c_1\in\partial\Delta_{\alpha_0}^{0}$ and hence $\alpha_0\in\Gamma$.
\end{proof}

Let $h_\alpha^0$ and $h_\alpha^\infty$ be the conformal mapping defined in \eqref{equ:h-0-infty}. They can be homeomorphically extended to
\begin{equation}\label{equ:h-0-infty-extend}
h_{\alpha}^0: \overline{\D} \to \overline{\Delta_{\alpha}^0}  \text{\quad and\quad} h_{\alpha}^\infty: \overline{\D} \to \overline{\Delta_{\alpha}^\infty}.
\end{equation}
For $\alpha\in\Gamma$, the map $f_\alpha$ has two critical points $c_1\in\partial\Delta_\alpha^0$ and $c_2\in\partial\Delta_\alpha^\infty$, and two critical values $f_\alpha(c_1)\in\partial\Delta_\alpha^\infty$ and $f_\alpha(c_2)\in\partial\Delta_\alpha^0$.
In this section we assume that $h_{\alpha}^0$ and $h_{\alpha}^\infty$ are normalized (which are different from \S\ref{subsec:continuity}) such that
\begin{equation}
h_{\alpha}^0(0)=0,~h_{\alpha}^0(1)=c_1 \text{\quad and\quad} h_{\alpha}^\infty(0)=\infty,~h_{\alpha}^\infty(1)=f_\alpha(c_1).
\end{equation}

\medskip
Note that $f_\alpha\circ h_\alpha^0:\overline{\D}\to\overline{\Delta_\alpha^\infty}$ is the homeomorphic extension of the conformal map $f_\alpha\circ h_\alpha^0:\D\to\Delta_\alpha^\infty$ satisfying $f_\alpha\circ h_\alpha^0(0)=\infty$ and $f_\alpha\circ h_\alpha^0(1)=f_\alpha(c_1)$. By the uniqueness of Riemann maps, we have $f_\alpha\circ h_\alpha^0(\zeta)=h_\alpha^\infty(\zeta)$ for all $\zeta\in\overline{\D}$.
Let
\begin{equation}\label{equ:A-A-tilde}
A(\alpha):=\arg\,(h_{\alpha}^0)^{-1}(f_\alpha(c_2)) \text{\quad and \quad} \widetilde{A}(\alpha):=\arg\,(h_{\alpha}^\infty)^{-1}(c_2)
\end{equation}
respectively, be the conformal angle between $c_1$ and $f_\alpha(c_2)$ (in the conformal coordinate $(h_{\alpha}^0)^{-1}$), and the conformal angle between $f_\alpha(c_1)$ and $c_2$ (in the conformal coordinate $(h_{\alpha}^\infty)^{-1}$) measured counterclockwise. Since $(h_\alpha^\infty)^{-1}\circ f_\alpha(z)=(h_\alpha^0)^{-1}(z)$ for all $z\in\overline{\Delta_\alpha^0}$, we have
\begin{equation}\label{equ:A-alpha}
\begin{split}
A(\alpha)
=&~\arg\,(h_{\alpha}^0)^{-1}\circ f_\alpha(c_2) \\
=&~\arg\,(h_{\alpha}^\infty)^{-1}\circ f_\alpha^{\circ 2}(c_2)=\widetilde{A}(\alpha)+2\pi\theta.
\end{split}
\end{equation}

\begin{lem}\label{continuousangle}
The map $\alpha\mapsto A(\alpha)$ is continuous on $\Gamma$.
\end{lem}

\begin{proof}
By Proposition \ref{continuousity}, $\partial\Delta_\alpha^0$ moves continuously as $\alpha$ varies continuously on $\Gamma$. Since the map $h_{\alpha}^0: \overline{\D} \to \overline{\Delta_{\alpha}^0}$ is normalized by $h_{\alpha}^0(0)=0$ and $h_{\alpha}^0(1)=c_1$, it follows from Carath\'{e}odory that $h_\alpha^0$ depends continuously on $\alpha\in\Gamma$. In particular, $\alpha\mapsto A(\alpha)$ is continuous on $\Gamma$.
\end{proof}

The next lemma implies that $\alpha\in\Gamma$ is uniquely determined by the angle $A(\alpha)$.

\begin{lem}\label{injection}
If $A({\alpha_1})=A({\alpha_2})$ for $\alpha_1,\alpha_2\in\Gamma$, then $\alpha_1=\alpha_2$.
\end{lem}

\begin{proof}
Based on the preparations in \S\ref{gamma}, we prove that $f_{\alpha_1}^{\circ 2}$ is conformally conjugate to $f_{\alpha_2}^{\circ 2}$ by a rigidity argument from $A(\alpha_1)=A(\alpha_2)$. To distinguish the objects corresponding to different parameters, we shall use $c_{1,\alpha}$ and $c_{2,\alpha}$ to denote the critical points of $f_\alpha$.

Since $A(\alpha_1)=A(\alpha_2)$, we have $\widetilde{A}(\alpha_1)=\widetilde{A}(\alpha_2)$ by \eqref{equ:A-alpha}. Define $\phi_0:=h_{\alpha_2}^0\circ (h_{\alpha_1}^0)^{-1}:\overline{\Delta_{\alpha_1}^0}\to\overline{\Delta_{\alpha_2}^0}$ and $\phi_0:=h_{\alpha_2}^\infty\circ (h_{\alpha_1}^\infty)^{-1}:\overline{\Delta_{\alpha_1}^\infty}\to\overline{\Delta_{\alpha_2}^\infty}$. Then
$$\phi_0: \overline{\Delta_{\alpha_1}^0} \cup \overline{\Delta_{\alpha_1}^{\infty}} \to \overline{\Delta_{\alpha_2}^0}\cup \overline{\Delta_{\alpha_2}^{\infty}}$$
is a continuous map satisfying
\begin{itemize}
\item $\phi_0: \Delta_{\alpha_1}^0\to \Delta_{\alpha_2}^0$ and $\phi_0: \Delta_{\alpha_1}^\infty\to \Delta_{\alpha_2}^\infty$ are conformal;
\item $\phi_0(c_{1,\alpha_1})=c_{1,\alpha_2}$ and $\phi_0(f_{\alpha_1}(c_{2,\alpha_1}))=f_{\alpha_2}(c_{2,\alpha_2})$;
\item $\phi_0(c_{2,\alpha_1})=c_{2,\alpha_2}$ and $\phi_0(f_{\alpha_1}(c_{1,\alpha_1}))=f_{\alpha_2}(c_{1,\alpha_2})$; and
\item $\phi_0\circ f_{\alpha_1}^{\circ 2}(z)=f_{\alpha_2}^{\circ 2}\circ\phi_0(z)$ for all $z\in \overline{\Delta_{\alpha_1}^0} \cup \overline{\Delta_{\alpha_1}^{\infty}}$.
\end{itemize}
We have the following:
\begin{claim}
The two maps $f_{\alpha_1}$ and $f_{\alpha_2}$ are Thurston equivalent. Specifically, there exist two quasiconformal mappings $\phi_0:\EC\to\EC$ and $\psi_0:\EC\to\EC$ such that
\begin{itemize}
  \item $\phi_0:\EC\to\EC$ and $\psi_0:\EC\to\EC$ are extensions of $\phi_0$ on $\overline{\Delta_{\alpha_1}^0} \cup \overline{\Delta_{\alpha_1}^{\infty}}$; and
  \item $\psi_0$ is isotopic to $\phi_0$ relative to $\partial{\Delta_{\alpha_1}^0} \cup \partial{\Delta_{\alpha_1}^{\infty}}$, and the following diagram is commutative:
\begin{equation}
\begin{CD}
\EC @>\psi_0>> \EC \\
@VVf_{\alpha_1} V   @VVf_{\alpha_2} V \\
\EC @>\phi_0>> \EC.
\end{CD}
\end{equation}
\end{itemize}
\end{claim}

In fact, since $\partial\Delta_\alpha^0$ and $\partial\Delta_\alpha^\infty$ are quasicircles for all $\alpha\in\Sigma_\theta$, there are quasiconformal mappings $\chi_i:\EC\to\EC$, where $i=1,2$, such that
\begin{itemize}
  \item $\chi_i:\Delta_{\alpha_i}^0\to\D_r$ and $\chi_i:\Delta_{\alpha_i}^\infty\to\EC\setminus\overline{\D}$ are conformal, where $\D_r=\{\zeta:|\zeta|<r\}$ for some $0<r<1$; and
  \item $\chi_i(c_{1,\alpha_i})=r$, $\chi_i\big(f_{\alpha_i}(c_{1,\alpha_i})\big)=1$, $\chi_i(0)=0$ and $\chi_i(\infty)=\infty$.
\end{itemize}
Then for $i=1,2$, each $\widetilde{f}_i:=\chi_i\circ f_{\alpha_i}\circ \chi_i^{-1}$ is a quasiregular map and
\begin{equation}
\widetilde{f}_i(\zeta)=r/\zeta: \overline{\D}_r\to \EC\setminus\D \text{\quad and\quad} \widetilde{f}_i(\zeta)=r e^{2\pi\ii\theta}/\zeta: \EC\setminus\D\to\overline{\D}_r.
\end{equation}

Let $U_0^0(\alpha_i)\neq \Delta_{\alpha_i}^0$ be the connected component of $f_{\alpha_i}^{-1}(\Delta_{\alpha_i}^\infty)$ attaching at $c_{1,\alpha_i}$. We use $\widetilde{U}_0^0(\alpha_i)\neq \Delta_{\alpha_i}^\infty$ to denote the connected component of $f_{\alpha_i}^{-1}(\Delta_{\alpha_i}^0)$ attaching at $c_{2,\alpha_i}$. Denote $\A_r=\{\zeta:r<|\zeta|<1\}$. Then $\widetilde{f}_i^{-1}(\A_r)=\A_r\setminus(\chi_i(U_0^0(\alpha_i)\cup\widetilde{U}_0^0(\alpha_i)))$ is an annulus and $\widetilde{f}_i:\widetilde{f}_i^{-1}(\A_r)\to\A_r$ is a covering map of degree two for $i=1,2$.
Hence $\widetilde{f}_2$ is homotopic to $T^{\circ k}\circ\widetilde{f}_1$ relative to $\partial\A_r$ for some $k\in\Z$, where
\begin{equation}
T(\zeta):=\left\{
\begin{array}{ll}
\zeta e^{2\pi\ii\frac{|\zeta|-r}{1-r}} & \text{\quad if } z\in\A_r,\\
\zeta &\text{\quad otherwise}.
\end{array}
\right.
\end{equation}
is the \textit{Dehn twist} in the annulus $\A_r$. See Figure \ref{Fig:Dehn-twist}.

\begin{figure}[!htpb]
  \setlength{\unitlength}{1mm}
  \centering
  \includegraphics[width=0.9\textwidth]{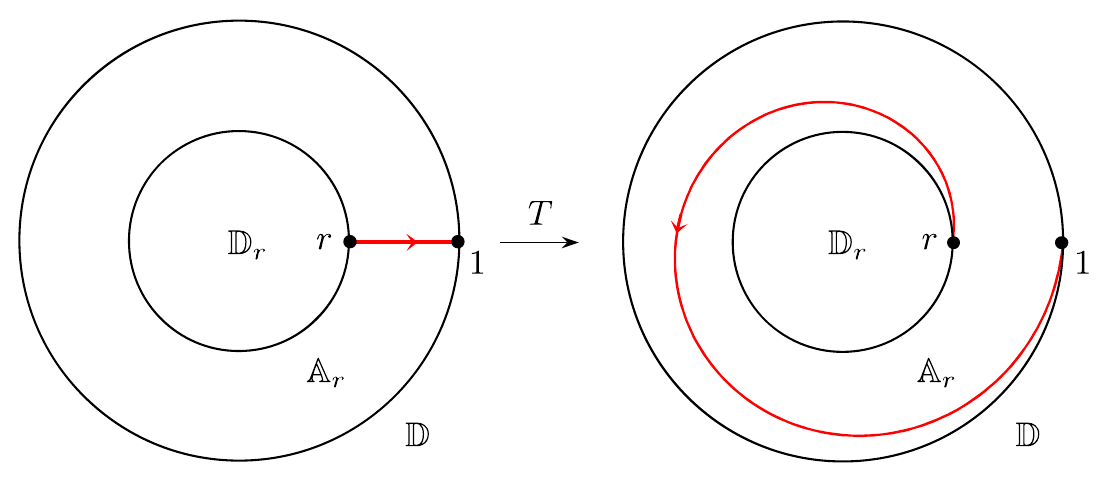}
  \caption{The Dehn twist $T$ in the annulus $\A_r=\D\setminus\overline{\D}_r$. The segment $[r,1]$ and its image under $T$ are colored red.}
  \label{Fig:Dehn-twist}
\end{figure}

For $i=1,2$, we denote
\begin{equation}
\widetilde{c}_{2,\alpha_i}:=\chi_i(c_{2,\alpha_i}) \text{\quad and\quad} \widetilde{v}_{2,\alpha_i}:=\widetilde{f}_i\big(\chi_i(c_{2,\alpha_i})\big)=\chi_i\big(f_i(c_{2,\alpha_i})\big).
\end{equation}
Since $\chi_i(c_{1,\alpha_i})=r$ and $\chi_i\big(f_{\alpha_i}(c_{1,\alpha_i})\big)=1$, by \eqref{equ:A-A-tilde} we have $A(\alpha_i)=\arg \widetilde{v}_{2,\alpha_i}$ and $\widetilde{A}(\alpha_i)=\arg \widetilde{c}_{2,\alpha_i}$.
Since $A(\alpha_1)=A(\alpha_2)$ and $\widetilde{A}(\alpha_1)=\widetilde{A}(\alpha_2)$, we have $\widetilde{v}_{2,\alpha_1}=\widetilde{v}_{2,\alpha_2}$ and $\widetilde{c}_{2,\alpha_1}=\widetilde{c}_{2,\alpha_2}$.

In the following, we use `$\simeq$' to denote the homotopy. Let $\widetilde{\phi}_0$ and $\widehat{\phi}_1$ be two quasiconformal mappings of $\EC$ such that
\begin{itemize}
\item $\widetilde{\phi}_0=\widehat{\phi}_1=\id$ in $\EC\setminus\A_r$; and
\item $\widetilde{\phi}_0\simeq \widehat{\phi}_1\simeq T^{-k}:\A_r\to\A_r$ rel $\partial\A_r$.
\end{itemize}
Note that for $i=1,2$, $\partial \A_r\subset\widetilde{f}_i^{-1}(\partial\A_r)$ and $\widetilde{f}_i:\widetilde{f}_i^{-1}(\A_r)\to\A_r$ is a covering map of degree two. From $\widetilde{f}_2 \simeq T^{\circ k}\circ\widetilde{f}_1$ rel $\partial\A_r$ for some $k\in\Z$, we have
\begin{equation}
\widetilde{f}_2\circ\widehat{\phi}_1\simeq T^{\circ k}\circ\widetilde{f}_1\circ \widehat{\phi}_1 \simeq T^{-k}\circ\widetilde{f}_1 \simeq \widetilde{\phi}_0\circ \widetilde{f}_1 \text{\quad rel } \partial\A_r.
\end{equation}
Hence there exists a quasiconformal mapping $\widetilde{\phi}_1:\EC\to\EC$ such that $\widetilde{\phi}_1=\id$ in $\EC\setminus\A_r$, $\widetilde{\phi}_1\simeq T^{-k}\simeq \widetilde{\phi}_0$ rel $\partial\A_r$, and the following diagram is commutative:
\begin{equation}
\begin{CD}
\EC @>\widetilde{\phi}_1>> \EC \\
@VV\widetilde{f}_1 V   @VV\widetilde{f}_2 V \\
\EC @>\widetilde{\phi}_0>> \EC.
\end{CD}
\end{equation}
See Figure \ref{Fig:twist-lift} for an illustration when $k=1$.
Then the claim that $f_{\alpha_1}$ and $f_{\alpha_2}$ are Thurston equivalent holds if we set $\phi_0:=\chi_2^{-1}\circ\widetilde{\phi}_0\circ\chi_1$ and $\psi_0:=\chi_2^{-1}\circ\widetilde{\phi}_1\circ\chi_1$.

\begin{figure}[!htpb]
  \setlength{\unitlength}{1mm}
  \centering
  \includegraphics[width=0.9\textwidth]{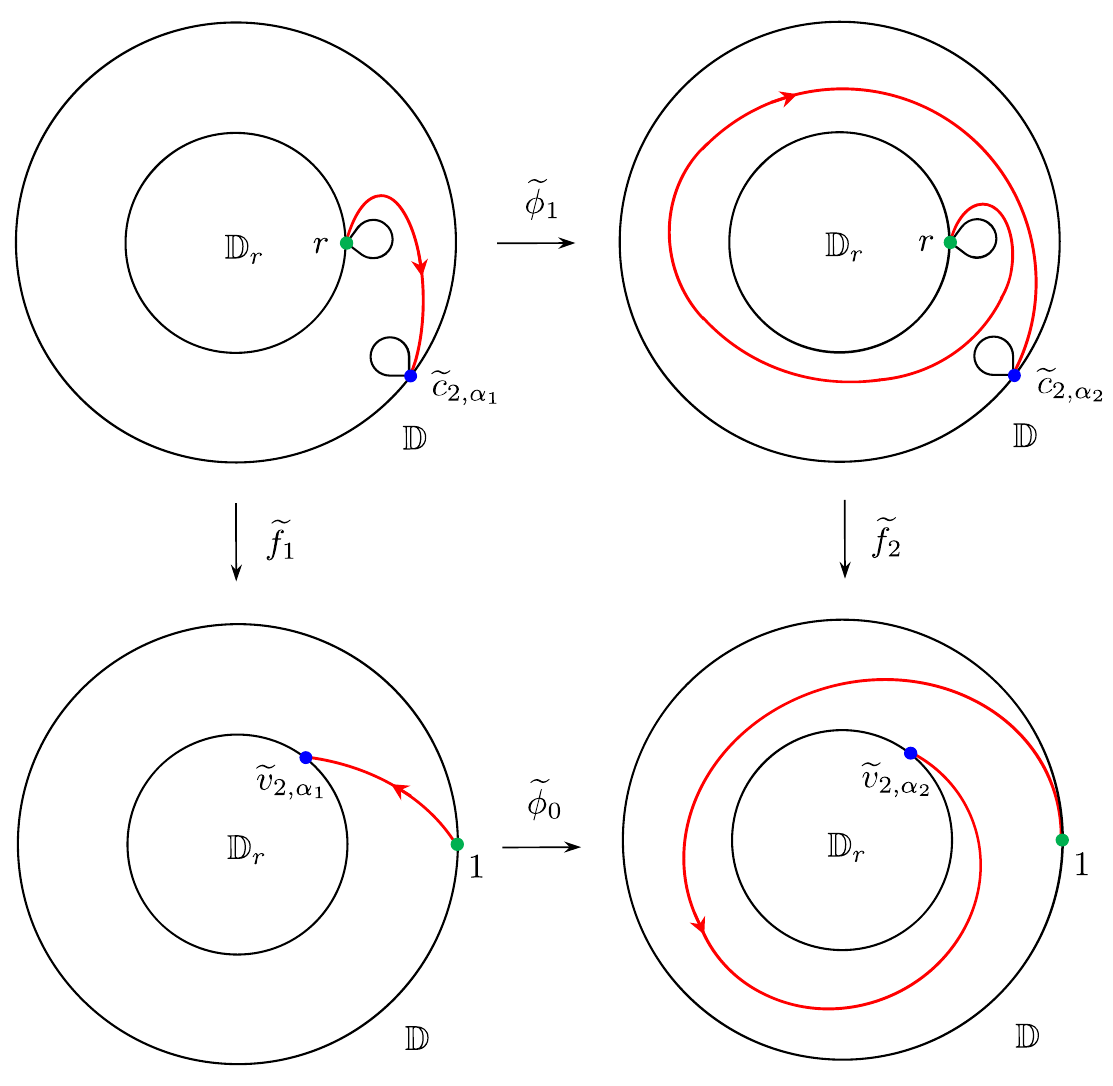}
  \caption{An illustration of the construction of the quasiconformal mappings $\widetilde{\phi}_0$ and $\widetilde{\phi}_1$ of $\EC$ under the condition $\widetilde{f}_2\simeq T^{\circ k}\circ\widetilde{f}_1$ rel $\partial\A_r$ when $k=1$.}
  \label{Fig:twist-lift}
\end{figure}

\medskip

Note that $f_{\alpha_i}^{-2}(\Delta_{\alpha_i}^0)$ is the union of $4$ connected components $\Delta_{\alpha_i}^0$ and $U_0^t(\alpha_i)$ with $t\in\{0,1,2\}$. Similarly, we use $\widetilde{U}_0^t(\alpha_i)$ with $t\in\{0,1,2\}$ to denote the components of $f_{\alpha_i}^{-2}(\Delta_{\alpha_i}^\infty)\setminus\Delta_{\alpha_i}^\infty$.
From the above claim we know that $\psi_0$ is conformal in every component of $f_{\alpha_1}^{-1}(\Delta_{\alpha_1}^0\cup\Delta_{\alpha_1}^\infty)=\Delta_{\alpha_1}^0\cup\Delta_{\alpha_1}^\infty\cup U_0^0(\alpha_1)\cup \widetilde{U}_0^0(\alpha_1)$, and there exists a quasiconformal mapping $\phi_1:\EC\to\EC$ which is the lift of $\psi_0$ such that
\begin{itemize}
  \item $\phi_1=\phi_0$ on $\overline{\Delta_{\alpha_1}^0} \cup \overline{\Delta_{\alpha_1}^{\infty}}$;
  \item $\phi_1\simeq\phi_0$ rel $\partial{\Delta_{\alpha_1}^0} \cup \partial{\Delta_{\alpha_1}^{\infty}}$;
  \item $\phi_0\circ f_{\alpha_1}^{\circ 2}=f_{\alpha_2}^{\circ 2}\circ\phi_1$; and
  \item $\phi_1$ is holomorphic in every component of $f_{\alpha_1}^{-2}(\Delta_{\alpha_1}^0\cup\Delta_{\alpha_1}^\infty)$, and maps $U_0^t(\alpha_1)$ to $U_0^t(\alpha_2)$ and $\widetilde{U}_0^t(\alpha_1)$  to $\widetilde{U}_0^t(\alpha_2)$ for $t\in\{0,1,2\}$.
\end{itemize}

\medskip
From \S\ref{subsec:address} we know that every connected component of $f_{\alpha_i}^{-2n}(\Delta_{\alpha_i}^0)$ which is different from $\Delta_{\alpha_i}^0$, can be written as exactly one of $U_{s_1,\cdots,s_m}^{t_1,\cdots,t_m}(\alpha_i)$, $V_{s_1,\cdots,s_m}^{t_1,\cdots,t_m}(\alpha_i)$, $W_{s_1,\cdots,s_m}^{t_1,\cdots,t_m}(\alpha_i)$ or $X_{s_1,\cdots,s_m}^{t_1,\cdots,t_m}(\alpha_i)$, where $n=s_1+\cdots +s_m+1$, $s_i\geq 1$ and $t_i\in\{0,1,2\}$. Similarly, we use $\widetilde{U}_{s_1,\cdots,s_m}^{t_1,\cdots,t_m}(\alpha_i)$, $\widetilde{V}_{s_1,\cdots,s_m}^{t_1,\cdots,t_m}(\alpha_i)$, $\widetilde{W}_{s_1,\cdots,s_m}^{t_1,\cdots,t_m}(\alpha_i)$ and $\widetilde{X}_{s_1,\cdots,s_m}^{t_1,\cdots,t_m}(\alpha_i)$ to denote the connected components of $f_{\alpha_i}^{-2n}(\Delta_{\alpha_i}^\infty)$.

\medskip
Now let us assume that for every $1\le k\le n$, we have a quasiconformal homeomorphism $\phi_k :\EC \to \EC$ so that
\begin{itemize}
  \item $\phi_k=\phi_{k-1}$ on $f_{\alpha_1}^{-2(k-1)}(\overline{\Delta_{\alpha_1}^0} \cup \overline{\Delta_{\alpha_1}^{\infty}})$;
  \item $\phi_k\simeq \phi_{k-1}$ rel $\partial{\Delta_{\alpha_1}^0} \cup \partial{\Delta_{\alpha_1}^{\infty}}$;
  \item $\phi_{k-1}\circ f_{\alpha_1}^{\circ 2}=f_{\alpha_2}^{\circ 2}\circ\phi_k$; and
  \item $\phi_k$ is holomorphic in every connected component of $f_{\alpha_1}^{-2k}(\Delta_{\alpha_1}^0\cup\Delta_{\alpha_1}^\infty)$, and maps every component to the corresponding one with the same address, i.e., $U_{s_1,\cdots,s_m}^{t_1,\cdots,t_m}(\alpha_1)$ to $U_{s_1,\cdots,s_m}^{t_1,\cdots,t_m}(\alpha_2)$ etc, where $k=s_1+\cdots +s_m+1$.
\end{itemize}

We define $\phi_{n+1}$ as follows. First we set $\phi_{n+1}=\phi_n$ on $f_{\alpha_1}^{-2n}(\overline{\Delta_{\alpha_1}^0} \cup \overline{\Delta_{\alpha_1}^{\infty}})$.
Note that $\bigcup_{0\leq k\leq n+1}f_{\alpha_i}^{-2k}(\overline{\Delta_{\alpha_i}^0})$ and $\bigcup_{0\leq k\leq n+1}f_{\alpha_i}^{-2k}(\overline{\Delta_{\alpha_i}^\infty})$ are connected and disjoint to each other, where $i=1,2$.
By induction, $\phi_{n+1}$ can be extended to a quasiconformal mapping $\phi_{n+1}:\EC\to\EC$ which satisfies
\begin{itemize}
  \item $\phi_{n+1}\simeq\phi_n$ rel $\partial{\Delta_{\alpha_1}^0} \cup \partial{\Delta_{\alpha_1}^{\infty}}$;
  \item $\phi_n\circ f_{\alpha_1}^{\circ 2}=f_{\alpha_2}^{\circ 2}\circ\phi_{n+1}$; and
  \item $\phi_{n+1}$ is holomorphic in every connected component of $f_{\alpha_1}^{-2(n+1)}(\Delta_{\alpha_1}^0\cup\Delta_{\alpha_1}^\infty)$, and maps every component to the corresponding one with the same address, i.e., $U_{s_1',\cdots,s_{m'}'}^{t_1',\cdots,t_{m'}'}(\alpha_1)$ to $U_{s_1',\cdots,s_{m'}'}^{t_1',\cdots,t_{m'}'}(\alpha_2)$ etc, where $n=s_1'+\cdots +s_{m'}'$.
\end{itemize}

By induction, we have a sequence of quasiconformal homeomorphisms $\{\phi_n:n\in\N\}$ of the complex sphere such that each $\phi_n:\EC\to\EC$ is conformal in every component of $f^{-2n}_{\alpha_1}(\Delta_{\alpha_1}^0 \cup \Delta_{\alpha_1}^{\infty})$ and its Beltrami coefficient satisfies
\begin{equation}
\| \mu_{\phi_n}\|_{\infty} \le \| \mu_{\phi_0}\|_{\infty}<1.
\end{equation}
Since $\phi_n=\phi_0$ on $\Delta_{\alpha_1}^0\cup\Delta_{\alpha_1}^\infty$ for all $n\geq 1$, it follows that $\{\phi_n:n\in\N\}$ is a normal family.

Passing to the convergent subsequences of $\{\phi_n\}_{n\geq 0}$ two times, we obtain two limit quasiconformal mappings $\phi$ and $\psi$, which fix $0$, $c_1$, $c_2$ and $\infty$, and satisfy $\phi\circ f_{\alpha_1}^{\circ 2}=f_{\alpha_2}^{\circ 2}\circ\psi$ since $\phi_n\circ f_{\alpha_1}^{\circ 2}=f_{\alpha_2}^{\circ 2}\circ\phi_{n+1}$ for each $n\geq 0$. From the construction it follows that $\phi(z)=\psi(z)$ for $z\in\bigcup_{n\geq 0}f_{\alpha_1}^{-2n}(\Delta_{\alpha_1}^0\cup\Delta_{\alpha_1}^\infty)$. Then $\phi=\psi$ on a dense set of $\EC$ since the Fatou set of $f_{\alpha_1}^{\circ 2}$ is dense on $\EC$. This implies that $\phi$ coincides with $\psi$ on $\EC$ because of the continuity. Then $f_{\alpha_1}^{\circ 2}$ and $f_{\alpha_2}^{\circ 2}$ are quasiconformally conjugate to each other on $\EC$ and they are conformally conjugate in the Fatou set.

By Lemma \ref{Juliaset}, the Julia set of $f_{\alpha_1}$ has zero Lebesgue measure. Hence $\phi:\EC\to\EC$ is conformal everywhere and must be the identity. This implies that $f^{\circ 2}_{\alpha_1}=f^{\circ 2}_{\alpha_2}$. By \eqref{equ:1}, we get $\alpha_1=\alpha_2$.
\end{proof}

Note that $A(\alpha)=0$ if and only if $f_\alpha(c_2)=c_1$. This is equivalent to $\alpha=\alpha_*:=c_1/f_1(c_2)$.
Hence we identify the two choices of images $A(\alpha_*)=0$ and $A(\alpha_*)=2\pi$ and assume that $A(\alpha)\in(0,2\pi)$ for all $\alpha\in\Gamma\setminus\{\alpha_*\}$.

\begin{lem}\label{continuousmap}
The map $A: \Gamma\to\R/(2\pi\Z)$ is a homeomorphism.
\end{lem}

\begin{proof}
By Proposition \ref{continuousity} and a similar proof to Lemma \ref{lem:alpha-0},  the set $\Gamma$ is compact in $\C\setminus\{0\}$. By Lemmas \ref{continuousangle} and \ref{injection}, $A$ is a continuous injection. Since $\R/(2\pi\Z)$ is a Hausdorff space, it suffices to prove that $A$ is a surjection.

Assume that $A(\Gamma)$ is a proper subset of $\R/(2\pi\Z)$. Since $\Gamma$ is compact and $A:\Gamma\to A(\Gamma)$  is a homeomorphism, it follows that $A(\Gamma)$ is a compact subset of $\R/(2\pi\Z)$ and every component of $A(\Gamma)$ is a singleton or a closed arc. Considering the homeomorphism $A^{-1}:A(\Gamma)\to\Gamma$, we know that every component of $\Gamma$ is a singleton or a simple arc with two end points.

By Lemma \ref{lem:alpha-0}, $\Gamma$ is a compact set separating $0$ from $\infty$. Let $\Lambda_0$ be the connected component of $\C\setminus(\Gamma\cup\{0\})$ containing a small punctured neighborhood of $0$. Then $\overline{\Lambda}_0$ is a connected compact set and $\partial\overline{\Lambda}_0\subset\Gamma$.
We claim that $\partial\overline{\Lambda}_0$ is connected. Otherwise, $\partial\overline{\Lambda}_0$ has at least two components, and each of them is contained in a singleton or a simple arc since $\partial\overline{\Lambda}_0\subset\Gamma$. This is impossible by the definition of $\overline{\Lambda}_0$. Hence $\partial\overline{\Lambda}_0$ is a connected compact set separating $0$ from $\infty$. However, still by $\partial\overline{\Lambda}_0\subset\Gamma$, we conclude that $\partial\overline{\Lambda}_0$ is contained in singleton or a simple arc, which cannot separate $0$ from $\infty$.  Hence we have $A(\Gamma)=\R/(2\pi\Z)$ and $A$ is a homeomorphism.
\end{proof}

\begin{proof}[{Proof of the Main Theorem}]
By Lemma \ref{lem:only-one} and Proposition \ref{continuousity}, each of $\partial\Delta_\alpha^0$ and $\partial\Delta_\alpha^\infty$ contains at most one critical point, and they move continuously as $\alpha$ varies continuously in $\Sigma_\theta$. It suffices to prove the properties of $\Gamma$.

By Lemma \ref{continuousmap}, $\Gamma$ is a Jordan curve in $\Sigma_\theta$ separating $0$ from $\infty$. For $\alpha\in\Sigma_\theta\setminus\Gamma$, $\partial\Delta_\alpha^0\cup\partial\Delta_\alpha^\infty$ contains exactly one critical point $c_1$ or $c_2$.
We have assumed that the critical points $c_1=c_1(\theta)$ and $c_2=c_2(\theta)$ are marked such that $c_1\in\partial\Delta_{\alpha}^0$ if $\alpha$ is large enough while $c_2\in \partial\Delta_{\alpha}^\infty$ if $\alpha$ is small enough.
In particular, by Corollary \ref{cor:Siegel-crit}, $c_1\in\partial\Delta_\alpha^0$ if $|\alpha|\geq\delta$ and $c_2\in\partial\Delta_\alpha^\infty$ if $0<|\alpha|\leq 1/\delta$ for some $\delta>1$.

We claim that $c_1\in\partial\Delta_{\alpha}^0$ for all $\alpha\in\Gamma_{\Ext}$. Otherwise, assume that $c_1\not\in\partial\Delta_{\alpha'}^0$ for some $\alpha'\in\Gamma_{\Ext}$.
Since $c_1\not\in\partial\Delta_{\alpha}^\infty$ and $c_2\not\in\partial\Delta_{\alpha}^0$ for all $\alpha\in\Sigma_\theta$, this implies that $c_2\in\partial\Delta_{\alpha'}^\infty$.
Let $\eta$ be a simple curve in $\Gamma_{\Ext}$ connecting $\delta$ with $\alpha'$. Since the boundaries of Siegel disks move continuously, similar to the proof of Lemma \ref{lem:alpha-0}, there must exist a point $\alpha''\in\eta$ such that $\partial\Delta_{\alpha''}^0\cup\partial\Delta_{\alpha''}^\infty$ contains the both critical points $c_1$ and $c_2$. This implies that $\alpha''\in\Gamma$, which is a contradiction since $\alpha''\in\eta\subset\Gamma_{\Ext}$. This proves the claim.
By a similar argument, we have $c_2\in\partial\Delta_{\alpha}^\infty$ for all $\alpha\in\Gamma_{\Int}\setminus\{0\}$.
This completes the proof of the Main Theorem.
\end{proof}

\bibliographystyle{amsalpha}
\bibliography{E:/Latex-model/Ref1}

\end{document}